\documentclass[11pt, a4paper, twoside,leqno]{amsart}
\usepackage[centering, totalwidth = 380pt, totalheight = 600pt]{geometry}
\usepackage{amssymb, amsmath, amsthm, listings, xspace}
\usepackage{enumerate, microtype, stmaryrd, url} 
\usepackage[latin1]{inputenc}
\usepackage[arrow, matrix, tips, curve, graph, rotate]{xy}
\usepackage[retainorgcmds]{IEEEtrantools}
 \usepackage{color}
\definecolor{daRemarkgreen}{rgb}{0,0.45,0}
\usepackage[colorlinks,citecolor=daRemarkgreen,linkcolor=daRemarkgreen]{hyperref}
\usepackage{lscape}
\SelectTips{cm}{10}

\newcommand{\cat}[1]{\mathbf{#1}}

\DeclareMathAlphabet      {\mathbf}{OT1}{cmr}{b}{n}

\newcommand{\cd}[2][]{\vcenter{\hbox{\xymatrix#1{#2}}}}

\makeatletter
\def\matrixobject@{%
   \edef \next@{={\DirectionfromtheDirection@ }}%
   \expandafter \toks@ \next@ \plainxy@
   \let\xy@@ix@=\xyq@@toksix@
   \xyFN@ \OBJECT@}
\let\xy@entry@@norm=\entry@@norm
\def\entry@@norm@patched{%
   \let\object@=\matrixobject@
   \xy@entry@@norm }
\AtBeginDocument{\let\entry@@norm\entry@@norm@patched}
\makeatother

\renewcommand{\phi}{\varphi}
\newcommand{\A}{{\mathcal A}}
\newcommand{\B}{{\mathcal B}}
\newcommand{\C}{{\mathcal C}}
\newcommand{\D}{{\mathcal D}}

\newcommand{\W}{{\mathcal W}}

\newcommand{\xtor}[1]{\cdl[@1]{{} \ar[r]|-{\object@{|}}^{#1} & {}}}

\makeatletter

\newcommand{\setmanuallabel}[1]{\stepcounter{equation}{\edef\@currentlabel{\theequation}\label{#1}}}
\newcommand{\printmanuallabel}[1]{\stepcounter{equation}\text{(\theequation)}}

\def\hookleftarrowfill@{\arrowfill@\leftarrow\relbar{\relbar\joinrel\rhook}}
\def\twoheadleftarrowfill@{\arrowfill@\twoheadleftarrow\relbar\relbar}
\def\leftbararrowfill@{\arrowdoublefill@{\leftarrow\mkern-5mu}\relbar\mapstochar\relbar\relbar}
\def\Leftbararrowfill@{\arrowdoublefill@{\Leftarrow\mkern-2mu}\Relbar\Mapstochar\Relbar\Relbar}
\def\leftringarrowfill@{\arrowdoublefill@{\leftarrow\mkern-3mu}\relbar{\mkern-3mu\circ\mkern-2mu}\relbar\relbar}
\def\lefttriarrowfill@{\arrowfill@{\mathrel\triangleleft\mkern0.5mu\joinrel\relbar}\relbar\relbar}
\def\Lefttriarrowfill@{\arrowfill@{\mathrel\triangleleft\mkern1mu\joinrel\Relbar}\Relbar\Relbar}

\def\hookrightarrowfill@{\arrowfill@{\lhook\joinrel\relbar}\relbar\rightarrow}
\def\twoheadrightarrowfill@{\arrowfill@\relbar\relbar\twoheadrightarrow}
\def\rightbararrowfill@{\arrowdoublefill@{\relbar\mkern-0.5mu}\relbar\mapstochar\relbar\rightarrow}
\def\Rightbararrowfill@{\arrowdoublefill@{\Relbar\mkern-2mu}\Relbar\Mapstochar\Relbar\Rightarrow}
\def\rightringarrowfill@{\arrowdoublefill@\relbar\relbar{\mkern-2mu\circ\mkern-3mu}\relbar{\mkern-3mu\rightarrow}}
\def\righttriarrowfill@{\arrowfill@\relbar\relbar{\relbar\joinrel\mkern0.5mu\mathrel\triangleright}}
\def\Righttriarrowfill@{\arrowfill@\Relbar\Relbar{\Relbar\joinrel\mkern1mu\mathrel\triangleright}}

\def\leftrightarrowfill@{\arrowfill@\leftarrow\relbar\rightarrow}
\def\mapstofill@{\arrowfill@{\mapstochar\relbar}\relbar\rightarrow}

\renewcommand*\xleftarrow[2][]{\ext@arrow 20{20}0\leftarrowfill@{#1}{#2}}
\providecommand*\xLeftarrow[2][]{\ext@arrow 60{22}0{\Leftarrowfill@}{#1}{#2}}
\providecommand*\xhookleftarrow[2][]{\ext@arrow 10{20}0\hookleftarrowfill@{#1}{#2}}
\providecommand*\xtwoheadleftarrow[2][]{\ext@arrow 60{20}0\twoheadleftarrowfill@{#1}{#2}}
\providecommand*\xleftbararrow[2][]{\ext@arrow 10{22}0\leftbararrowfill@{#1}{#2}}
\providecommand*\xLeftbararrow[2][]{\ext@arrow 50{24}0\Leftbararrowfill@{#1}{#2}}
\providecommand*\xleftringarrow[2][]{\ext@arrow 10{26}0\leftringarrowfill@{#1}{#2}}
\providecommand*\xlefttriarrow[2][]{\ext@arrow 80{24}0\lefttriarrowfill@{#1}{#2}}
\providecommand*\xLefttriarrow[2][]{\ext@arrow 80{24}0\Lefttriarrowfill@{#1}{#2}}

\renewcommand*\xrightarrow[2][]{\ext@arrow 01{20}0\rightarrowfill@{#1}{#2}}
\providecommand*\xRightarrow[2][]{\ext@arrow 04{22}0{\Rightarrowfill@}{#1}{#2}}
\providecommand*\xhookrightarrow[2][]{\ext@arrow 00{20}0\hookrightarrowfill@{#1}{#2}}
\providecommand*\xtwoheadrightarrow[2][]{\ext@arrow 03{20}0\twoheadrightarrowfill@{#1}{#2}}
\providecommand*\xrightbararrow[2][]{\ext@arrow 01{22}0\rightbararrowfill@{#1}{#2}}
\providecommand*\xRightbararrow[2][]{\ext@arrow 04{24}0\Rightbararrowfill@{#1}{#2}}
\providecommand*\xrightringarrow[2][]{\ext@arrow 01{26}0\rightringarrowfill@{#1}{#2}}
\providecommand*\xrighttriarrow[2][]{\ext@arrow 07{24}0\righttriarrowfill@{#1}{#2}}
\providecommand*\xRighttriarrow[2][]{\ext@arrow 07{24}0\Righttriarrowfill@{#1}{#2}}

\providecommand*\xmapsto[2][]{\ext@arrow 01{20}0\mapstofill@{#1}{#2}}
\providecommand*\xleftrightarrow[2][]{\ext@arrow 10{22}0\leftrightarrowfill@{#1}{#2}}
\providecommand*\xLeftrightarrow[2][]{\ext@arrow 10{27}0{\Leftrightarrowfill@}{#1}{#2}}

\makeatother

\newcommand{\twocong}[2][0.5]{\ar@{}[#2] \save ?(#1)*{\cong}\restore}
\newcommand{\twoeq}[2][0.5]{\ar@{}[#2] \save ?(#1)*{=}\restore}
\newcommand{\rtwocell}[3][0.5]{\ar@{}[#2] \ar@{=>}?(#1)+/l 0.2cm/;?(#1)+/r 0.2cm/^{#3}}
\newcommand{\ltwocell}[3][0.5]{\ar@{}[#2] \ar@{=>}?(#1)+/r 0.2cm/;?(#1)+/l 0.2cm/^{#3}}
\newcommand{\ltwocello}[3][0.5]{\ar@{}[#2] \ar@{=>}?(#1)+/r 0.2cm/;?(#1)+/l 0.2cm/_{#3}}
\newcommand{\dtwocell}[3][0.5]{\ar@{}[#2] \ar@{=>}?(#1)+/u  0.2cm/;?(#1)+/d 0.2cm/^{#3}}
\newcommand{\dltwocell}[3][0.5]{\ar@{}[#2] \ar@{=>}?(#1)+/ur  0.2cm/;?(#1)+/dl 0.2cm/^{#3}}
\newcommand{\drtwocell}[3][0.5]{\ar@{}[#2] \ar@{=>}?(#1)+/ul  0.2cm/;?(#1)+/dr 0.2cm/^{#3}}
\newcommand{\dthreecell}[3][0.5]{\ar@{}[#2] \ar@3{->}?(#1)+/u  0.2cm/;?(#1)+/d 0.2cm/^{#3}}
\newcommand{\utwocell}[3][0.5]{\ar@{}[#2] \ar@{=>}?(#1)+/d 0.2cm/;?(#1)+/u 0.2cm/_{#3}}
\newcommand{\dtwocelltarg}[3][0.5]{\ar@{}#2 \ar@{=>}?(#1)+/u  0.2cm/;?(#1)+/d 0.2cm/^{#3}}
\newcommand{\utwocelltarg}[3][0.5]{\ar@{}#2 \ar@{=>}?(#1)+/d  0.2cm/;?(#1)+/u 0.2cm/_{#3}}

\newdir{(}{{}*!<0em,-.14em>-\cir<.14em>{l^r}}
\newdir{ (}{{}*!/-5pt/\dir{(}}
\newdir{ >}{{}*!/-5pt/\dir{>}}

\newcommand{\h}[1]{#1}

\theoremstyle{definition}
\newtheorem*{AxiomM}{Axiom M}
\newtheorem*{AxiomC}{Axiom C}
\newtheorem*{AxiomMC}{Axiom MC}

\numberwithin{equation}{section}
\theoremstyle{plain}
\newtheorem{Theorem}{Theorem}[section]
\newtheorem*{thank}{Acknowledgments}

\newtheorem{Proposition}[Theorem]{Proposition}
\newtheorem{Lemma}[Theorem]{Lemma}

\theoremstyle{definition}

\newtheorem{Definition}[Theorem]{Definition}
\newtheorem{Not}[Theorem]{Notation}
\newtheorem{Var}[Theorem]{Variants}

\newtheorem{Remark}[Theorem]{Remark}

\newcommand{\f}[1]{\mathbf #1}
\newcommand{\atwo}{{\mathbf 2}}
\newcommand{\Perm}{\textnormal{Perm}\xspace}
\newcommand{\SMon}{\textnormal{SMon}\xspace}

\newcommand{\DLim}{\mathbf{D}\textnormal{-Lim}}
\newcommand{\DLIM}{\mathbf{D}\textnormal{-}\mathbf{Lim}}

\newcommand{\Cat}{{\cat{Cat}}}
\newcommand{\CatGph}{{\cat{Cat}\textnormal{-Gph}}}
\newcommand{\CAT}{{\cat{CAT}}}
\newcommand{\Set}{{\cat{Set}}}

\newcommand{\TALG}{\mathbf{T}\textnormal{-}\mathbf{Alg}}
\newcommand{\TAlg}{\textnormal{T-Alg}}
\newcommand{\TAlgs}{\textnormal{T-Alg}_{s}}
\newcommand{\TAlgsc}{\textnormal{(T-Alg}_{s})_{c}}

\newcommand{\twocat}{2\textnormal{-Cat}}
\newcommand{\TWOCAT}{\mathbf{2}\textnormal{-}\mathbf{Cat}}
\newcommand{\Bicat}{\textnormal{Bicat}}
\newcommand{\Icon}{\textnormal{Icon}}
\newcommand{\BICAT}{\mathbf{Bicat}}

\begin{document}
\leftmargini=2em \title{Skew structures in 2-category theory and homotopy theory}
\author{John Bourke}
\address{Department of Mathematics and Statistics, Masaryk University, Kotl\'a\v rsk\'a 2, Brno 60000, Czech Republic}
\email{bourkej@math.muni.cz}
\date{\today}

\maketitle
\begin{abstract}
We study Quillen model categories equipped with a monoidal skew closed structure that descends to a genuine monoidal closed structure on the homotopy category.  Our examples are 2-categorical and include permutative categories and bicategories.  Using the skew framework, we adapt Eilenberg and Kelly's theorem relating monoidal and closed structure to the homotopical setting.  This is applied to the construction of monoidal bicategories arising from the pseudo-commutative 2-monads of Hyland and Power.
\end{abstract}

\section{Introduction}
The notion of a monoidal closed category captures the behaviour of the tensor product and internal hom on classical categories such as those of sets and vector spaces.  Some of the basic facts about monoidal closed categories have an intuitive meaning.  For instance, the isomorphism
\begin{equation}
\C(A,B) \cong \C(I,[A,B])
\end{equation}
says that elements of the internal hom $[A,B]$ are the same thing as morphisms $A \to B$.\\
Recently some new variants have come to light.  Firstly, the \emph{skew} monoidal categories of Szlach\'{a}nyi \cite{Szlachanyi2012Skew-monoidal} in which the structure maps such as $(A \otimes B) \otimes C \to A \otimes (B \otimes C)$ have a specified orientation and are not necessarily invertible. Shortly afterwards the dual notion of a skew closed category was introduced by Street \cite{Street2013Skew}.  Here one has a canonical map
\begin{equation}
\C(A,B) \to \C(I,[A,B])
\end{equation}
but this need not be invertible. Intuitively, we might view this relaxation as saying that $[A,B]$ should contain the morphisms $A \to B$ as elements, but possibly something else too.\\
In the present paper a connection is drawn between skew structures and homotopy theory.  We study examples of Quillen model categories $\C$ in which the correct internal homs $[A,B]$ have more general \emph{weak maps} $A \rightsquigarrow B$ as elements.  By the above reasoning these examples are necessarily skew. These skew closed categories form part of enveloping monoidal skew closed structures that descend to the homotopy category $Ho(\C)$ where, in fact, they yield \emph{genuine} monoidal closed structures.  The study of skew structures on a category that induce genuine structures on the homotopy category is our main theme.\\
Our examples are 2-categorical in nature -- most involve tweaking better known weak 2-categorical structures to yield not strict, but skew, structures.  For example, we describe a monoidal skew closed structure on the 2-category of permutative categories -- symmetric strict monoidal categories -- and strict maps.  This contains, on restricting to the cofibrant objects, a copy of the well known monoidal bicategory of permutative categories and strong maps.  More generally, we describe a skew structure for each pseudo-commutative 2-monad $T$ on $\Cat$ in the sense of \cite{Hyland2002Pseudo}.  Other examples concern 2-categories and bicategories.\\ 
The theory developed in the present paper has a future goal, concerning Gray-categories, in mind.  It was shown in \cite{Bourke2015A-cocategorical} that there exists no homotopically well behaved monoidal biclosed structure on the category of Gray-categories.  The plan is, in a future paper, to use the results developed here to understand the correct enriching structure on the category of Gray-categories.\\
Let us now give an overview of the paper.	Section 2 is mainly background on skew monoidal, skew closed and monoidal skew closed categories.  We recall Street's theorem describing the perfect correspondence between skew monoidal structures $(\C,\otimes,I)$ and skew closed structures $(\C,[-,-],I)$ in the presence of adjointness isomorphisms $\C(A \otimes B,C) \cong \C(A,[B,C])$.  In Theorem~\ref{Theorem:EK} we reformulate Eilenberg and Kelly's theorem \cite{Eilenberg1966Closed}, relating monoidal and closed structure, in the skew language.  Finally, we introduce symmetric skew closed categories.\\
It turns out that all the examples of skew closed structures that we meet in the present paper can be seen as arising from certain multicategories in a canonical way.  In Section 3 we describe the passage from such multicategories to skew closed categories.\\
Using the multicategory approach where convenient, Section 4 gives concrete examples of some of the skew closed structures that we are interested in.  We describe the examples of categories with limits, permutative categories, 2-categories and bicategories.\\
Section 5 concerns the interaction between skew structures and Quillen model structures that lies at the heart of the paper.  We begin by describing how a skew monoidal structure $(\C,\otimes,I)$ can be \emph{left derived} to the homotopy category.  This is the skew version of Hovey's construction \cite{Hovey1999Model}.  We call $(\C,\otimes,I)$ \emph{homotopy monoidal} if the left derived structure $(Ho(\C),\otimes_{l},I)$ is genuinely monoidal.  This is complemented by an analysis of how skew closed structure can be \emph{right derived} to the homotopy category, and we obtain a corresponding notion of \emph{homotopy closed} category.  Combining these cases Theorem~\ref{Theorem:TotalDerived} describes how monoidal skew closed structure can be \emph{derived} to the homotopy category.  This is used to prove Theorem~\ref{Theorem:hEK}, a homotopical analogue of Eilenberg and Kelly's theorem, which allows us to recognise homotopy monoidal structure in terms of homotopy closed structure.\\
Section 6 returns to the examples of categories with limits and permutative categories in the more general setting of pseudo-commutative 2-monads $T$ on $\Cat$.  We make minor modifications to Hyland and Power's construction \cite{Hyland2002Pseudo} of a pseudo-closed structure on $\TAlg$ to produce a skew closed structure on the 2-category $\TAlgs$ of algebras and strict morphisms.  For accessible $T$ this forms part of an enveloping monoidal skew closed structure which, using Theorem~\ref{Theorem:hEK}, we show to be homotopy monoidal.  Using this, we give a complete construction of the monoidal bicategory structure on $\TAlg$ of Hyland and Power -- thus solving a problem of \cite{Hyland2002Pseudo}.\\
Section 7 consists of an in-depth analysis of the skew structure on the category of bicategories and strict homomorphisms.  Though not particularly interesting in its own right, we regard this example as a preliminary to future work in higher dimensions.

\begin{thank}
The author thanks the organisers of the Cambridge Category Theory Seminar and of CT2015 in Aveiro for providing the opportunity to present this work, and thanks Sofie Royeaerd for useful feedback on an early draft.
\end{thank}

\section{Skew monoidal and skew closed categories}
\subsection{Skew monoidal categories}
Skew monoidal categories were introduced by Szlach\'{a}nyi \cite{Szlachanyi2012Skew-monoidal} in the study of bialgebroids over rings.  There are left and right versions (depending upon the orientation of the associativity and unit maps)
and it is the left handed case that is of interest to us.
\begin{Definition}
A (left) skew monoidal category $(\C,\otimes,I,\alpha,l,r)$ is a category $\C$ together with a functor $\otimes:\C \times \C \to \C$, a unit object $I \in \C$, and natural families $\alpha_{A,B,C}:(A \otimes B) \otimes C \to A \otimes (B \otimes C)$, $l_{A}:I \otimes A \to A$ and $r_{A}:A \to A \otimes I$ satisfying five axioms \cite{Szlachanyi2012Skew-monoidal}.
\end{Definition}
There is no need for us to reproduce these five axioms here as we will not use them.  We remark only that they are neatly labelled by the five words 
\begin{center}
$abcd$\\
$aib$ $aib$ $abi$ \\
$ii$
\end{center}
of which the first  refers to MacLane's pentagon axiom.\\
Henceforth the term skew monoidal is taken to mean left skew monoidal.  A monoidal category is precisely a skew monoidal category in which the constraints $\alpha, l$ and $r$ are invertible.
\subsection{Skew closed categories}
In the modern treatment of monoidal closed categories as a basis for enrichment \cite{Kelly1982Basic} it is the monoidal structure that is typically treated as primitive.  Nonetheless, the first major treatment \cite{Eilenberg1966Closed} emphasised the closed structure, presumably because internal homs are often more easily described than the corresponding tensor products.  In the examples of interest to us (see Section~\ref{section:examples}) this is certainly the case.  These examples will not be closed in the sense of \emph{ibid.}  but only skew closed.
\begin{Definition}[Street \cite{Street2013Skew}]
A (left) skew closed category $(\C,[-,-],I,L,i,j)$ consists of a category $\C$ equipped with a bifunctor $[-,-]:\C^{op} \times \C \to \C$ and unit object $I$ together with
\begin{enumerate}
\item components $L=L^{A}_{B,C}:[B,C] \to [[A,B],[A,C]]$ natural in $B,C$ and extranatural in $A$,
\item a natural transformation $i=i_{A}:[I,A] \to A$, 
\item components $j=j_{A}:I \to [A,A]$ extranatural in $A$,
\end{enumerate}
satisfying the following five axioms.
\begin{equation*}\tag{C1}
\begin{aligned}
\cd{& [[A,C],[A,D]] \ar[rd]^-{ L}  & \\
[C,D] \ar[ru]^-{L} \ar[d]_-{L} & & [[[A,B],[A,C]],[[A,B],[A,D]]] \ar[d]^-{[L,1]} \\
[[B,C],[B,D]] \ar[rr]_-{[1,L]} & & [[B,C],[[A,B],[A,D]]] }
\end{aligned}
\end{equation*}.

$$\xy
(0,0)*+{(C2)}="00"; (20,0)*+{ [[A,A],[A,C]]}="10"; (50,0)*+{[I,[A,C]]}="20";(70,0)*+{(C3)}="30"; (83,0)*+{[B,B]}="40"; (117,0)*+{[[A,B],[A,B]]}="50"; 
(20,-15)*+{[A,C]}="01"; (50,-15)*+{[A,C]}="11"; (100,-15)*+{I}="12";
{\ar^{1} "01"; "11"};{\ar^{L} "01"; "10"};{\ar^{i} "20"; "11"};
{\ar^{[j,1]} "10"; "20"};{\ar^{L} "40"; "50"};
{\ar^{j} "12"; "40"};{\ar_{j} "12"; "50"};
\endxy$$

$$\xy
(0,0)*+{(C4)}="00"; (15,0)*+{[B,C]}="10"; (50,0)*+{[[I,B],[I,C]]}="20";(69,0)*+{(C5)}="30"; (85,0)*+{I}="40"; (122,0)*+{[I,I]}="50"; 
(30,-15)*+{[[I,B],C]}="01"; (103,-15)*+{I}="11";
{\ar^{L} "10"; "20"};{\ar^{j} "40"; "50"};
{\ar_{[i,1]} "10"; "01"};{\ar^{[1,i]} "20"; "01"};
{\ar_{1} "40"; "11"};{\ar^{i} "50"; "11"};
\endxy$$
\end{Definition}

$(\C,[-,-],I)$ is said to be \emph{left normal} when the composite function 
\begin{equation}
\cd{\C{(A,B)} \ar[rr]^{[A,-]} & & \C({[A,A],[A,B])} \ar[rr]^{\C(j,1)} & & \C{(I,[A,B])}}
\end{equation}
is invertible, and \emph{right normal} if $i:[I,A] \to A$ is invertible.  A \emph{closed category} is, by definition, a skew closed category which is both left and right normal.\begin{footnote}{ The original definition of closed category \cite{Eilenberg1966Closed} involved an \emph{underlying functor} to $\Set$.  We are using the modified definition of \cite{Street1974Elementary} (see also \cite{Laplaza1977Embedding}) which eliminates the reference to $\Set$.}\end{footnote}
\begin{Var}\label{Theorem:Var1}
We will regularly mention a couple of variants on the above definition and we note them here.
\begin{enumerate}
\item We will sometimes consider skew closed 2-categories: the $\Cat$-enriched version of the above concept.  The difference is that $\C$ is now a \emph{2-category}, $[-,-]$ a \emph{2-functor} and each of the three transformations \emph{2-natural} in each variable.
\item We call a structure $(\C,[-,-],L)$ satisfying $C1$ but without unit a \emph{semi-closed} category.
\end{enumerate}
\end{Var}


\subsection{The correspondence between skew monoidal and skew closed categories}\label{section:correspondence}
A monoidal category $(\C,\otimes,I)$ in which each functor $- \otimes A:\C \to \C$ has a right adjoint $[A,-]$ naturally gives rise to the structure $(\C,[-,-],I)$ of a closed category.  Counterexamples to the converse statement are described in Section 3 of \cite{Day1978On}: no closed category axiom ensures the associativity of the corresponding tensor product.  An appealing feature of the skew setting is that there is a perfect correspondence between skew monoidal and skew closed structure.
\begin{Theorem}[Street \cite{Street2013Skew}]
Let $\C$ be a category equipped with an object $I$ and a pair of bifunctors $\otimes:\C \times \C \to \C$ and $[-,-]:\C^{op} \times \C \to \C$ related by isomorphisms $\phi:\C(A \otimes B,C) \cong \C(A,[B,C])$ natural in each variable.  There is a bijection between extensions of $(\C,\otimes,I)$ to a skew monoidal structure and of $(\C,[-,-],I)$ to a skew closed structure.
\end{Theorem}
Our interest is primarily in the passage from the closed to the monoidal side and, breaking the symmetry slightly, we describe it now: for the full symmetric treatment see \cite{Street2013Skew}.\\
\begin{itemize}
\item $l:I \otimes A \to A$ is the unique map such that the diagram
\begin{equation}\label{eq:leftunit}
\cd{\C(A,B) \ar[drr]_{v} \ar[rr]^{\C(l,1)} && \C(I \otimes A,B) \ar[d]^{\phi}\\
&& \C(I,[A,B])}
\end{equation}
commutes for all $B$.  Here $v=\C(j,1) \circ [A,-]:\C(A,B) \to \C(I,[A,B])$ is the morphism defining left normality.  \emph{In particular $l$ is invertible for each $A$ just when $v$ is.}
\item
$r:A \to A \otimes I$ is the unique morphism such that the diagram
\begin{equation}\label{eq:rightunit}
 \cd{\C(A \otimes I,B) \ar[d]_{\phi} \ar[rr]^{\C(r,1)} && \C(A,B)\\
\C(A,[I,B]) \ar[urr]_{\C(1,i)}}
\end{equation}
commutes for all $B$.  \emph{In particular $r$ is invertible for each $A$ just when $i$ is.}
\item Transposing the identity through the isomorphism $\phi:\C(A \otimes B,A \otimes B) \cong \C(A,[B,A \otimes B])$ yields a morphism $u:A \to [B,A \otimes B]$ natural in each variable.  Write $t:[A \otimes B,C] \to [A,[B,C]]$ for the composite
\begin{equation}\label{eq:t}
\cd{[A \otimes B,C] \ar[rr]^{L} && [[B,A \otimes B],[B,C]] \ar[rr]^{[u,1]} && [A,[B,C]]}
\end{equation}
which, we note, is natural in each variable.  The constraint 
$\alpha:(A \otimes B) \otimes C \to A \otimes (B \otimes C)$ is the unique morphism rendering commutative the diagram
\begin{equation}\label{eq:ass}
\cd{
\C(A \otimes (B \otimes C),D) \ar[dd]_{\phi} \ar[rr]^{\C(\alpha,1)} && \C((A \otimes B) \otimes C,D) \ar[d]^{\phi}\\
&& \C(A \otimes B,[C,D]) \ar[d]^{\phi}\\
\C(A,[B \otimes C,D]) \ar[rr]_{\C(1,t)} && \C(A,[B,[C,D]])
}
\end{equation}
for all $D$.  \emph{In particular $\alpha$ is invertible just when $t$ is.}
\end{itemize}
\begin{Definition}
A monoidal skew closed category consists of a skew monoidal category $(\C,\otimes,I,\alpha,l,r)$ and skew closed category $(\C,[-,-],I,L,i,j)$, together with natural isomorphisms $\phi:\C(A \otimes B,C) \cong \C(A,[B,C])$ all related by the above equations. 
\end{Definition}
Of course in the presence of the isomorphisms either bifunctor determines the other.  Accordingly a monoidal skew closed category is determined by either the skew monoidal or closed structure together with the isomorphisms $\phi$.\\
We remark that monoidal skew closed structures on the category of left $R$-modules over a ring $R$ that have $R$ as unit correspond to left bialgebroids over $R$.  This was the reason for the introduction of skew monoidal categories in \cite{Szlachanyi2012Skew-monoidal}.\\
The following result -- immediate from the above -- is, minus the skew monoidal terminology, contained within Chapter 2 and in particular Theorem 5.3 of \cite{Eilenberg1966Closed}.
\begin{Theorem}[Eilenberg-Kelly]\label{Theorem:EK}
Let $(\C,\otimes,[-,-],I)$ be a monoidal skew closed category. Then $(\C,\otimes,I)$ is monoidal if and only if $(\C, [-,-],I)$ is closed and the transformation $t:[A \otimes B,C] \to [A,[B,C]]$ 
is an isomorphism for all $A,B$ and $C$.
\end{Theorem}
Eilenberg and Kelly's theorem can be used to recognise monoidal structure in terms of closed structure.  However it can be difficult to determine whether $t:[A \otimes B,C] \to [A,[B,C]]$ is invertible.  This difficulty disappears in the presence of a suitable symmetry.

\subsection{Symmetry}
A symmetry on a skew closed category begins with a natural isomorphism $s:[A,[B,C]] \cong [B,[A,C]]$.  If $\C$ is left normal the vertical maps
\begin{equation*}
\cd{
\C(A,[B,C]) \ar[d]_{v} \ar[rr]^{s_{0}} && \C(B,[A,C]) \ar[d]^{v} \\
\C(I,[A,[B,C]]) \ar[rr]^{\C(I,s)} && \C(I,[B,[A,C]])}
\end{equation*}
are isomorphisms, so that we obtain an isomorphism $s_{0}$ by conjugating $\C(I,s)$.  If $\C$ underlies a monoidal skew closed category this in turn gives rise to a natural isomorphism $$\C(A \otimes B,C) \cong \C(B \otimes A,C)$$ and so, by Yoneda, a natural isomorphism 
\begin{equation}
c:B \otimes A \cong A \otimes B \hspace{0.3cm} .
\end{equation}
Our leading examples of skew closed categories do admit symmetries, but are \emph{not} left normal: accordingly, the symmetries are visible on the closed side but not on the monoidal side.  However they often reappear on the monoidal side upon passing to the \emph{homotopy category} -- see Theorem~\ref{Theorem:hEKs}.
\begin{Definition}
A symmetric skew closed category consists of a skew closed category $(\C,[-,-],I)$ together with a natural isomorphism $s:[A,[B,C]] \cong [B,[A,C]]$ satisfying the following four equations.
\begin{equation}\tag{S1}
\cd{[A,[B,C]] \ar[dr]_{s} \ar[rr]^{1} & & [A,[B,C]]\\
& [B,[A,C]] \ar[ur]_{s}}
\end{equation}
\begin{equation}\tag{S2}
\cd{[A,[B,[C,D]]] \ar[d]_{[1,s]} \ar[r]^{s} & [B,[A,[C,D]]] \ar[r]^{[1,s]} & [B,[C,[A,D]]] \ar[d]^{s}\\
[A,[C,[B,D]]] \ar[r]^{s} & [C,[A,[B,D]]] \ar[r]^{[1,s]} & [C,[B,[A,D]]]}
\end{equation}
\begin{equation}\tag{S3}
\cd{[A,[B,C]] \ar[d]_{s} \ar[r]^<<<{L} & [[D,A],[D,[B,C]]] \ar[r]^{[1,s]} & [[D,A],[B,[D,C]]]  \ar[d]^{s}\\
[B,[A,C]] \ar[rr]^{[1,L]} & & [B,[[D,A],[D,C]]]}
\end{equation}
\begin{equation}\tag{S4}
\cd{[A,B] \ar@/_2pc/[rrrrr]_{1} \ar[rr]^{L} & & [[A,A],[A,B]] \ar[r]^{[j,1]} & [I,[A,B]] \ar[r]^{s} & [A,[I,B]] \ar[r]^{[1,i]} & [A,B]}
\end{equation}
\end{Definition}
$\C$ is said to be symmetric closed if its underlying skew closed category is closed.  
\begin{Var}\label{Theorem:Var2}
As in Variants~\ref{Theorem:Var1} there are evident notions of symmetric skew closed 2-categories and symmetric semi-closed categories.
\end{Var}
\begin{Remark}
The notion of symmetric closed category described above coincides with that of \cite{Day1978On}, though this may not be immediately apparent.  Their invertible unit map $i:X \to [I,X]$ points in the opposite direction to ours.  Reversing it, their (CC4) is clearly equivalent to our (S4).  Their remaining axioms are a proper subset of those above, with (C1), (C3), (C4) and (C5) omitted.  But as they point out in Proposition 1.3 any symmetric closed category in their sense is a closed category and hence satisfies all four of these.
\end{Remark}
I first encountered a result close to the following one as Proposition 2.3 of \cite{Day1978On}, which shows that a symmetric closed category $\C$ gives rise to a symmetric promonoidal one by setting $P(A,B,C)=\C(A,[B,C])$.  This easily implies that a symmetric closed category gives rise to a symmetric monoidal one on taking adjoints.  In a discussion about that result, Ross Street pointed out that a skew monoidal category with an invertible natural isomorphism $A\otimes B \cong B \otimes A$ satisfying the braid equation $B5$ of \cite{Joyal1986Braided} is necessarily associative.  The first part of the following result essentially reformulates Street's associativity argument in terms of the closed structure.  Diagram 4.8 of Chapter IV of De Schipper's book \cite{deSchipper1975Symmetric} -- ~\eqref{eq:smc} below -- proved helpful in making that reformulation.

\begin{Theorem}[Day-LaPlaza, Street]\label{Theorem:ClassicalSymmetry}
Let $(\C,\otimes,[-,-],I)$ be monoidal skew closed.
\begin{enumerate}
\item The transformation $t:[A \otimes B,C] \to [A,[B,C]]$ is invertible if $(\C,[-,-],I)$ is left normal and admits a natural isomorphism $s:[A,[B,C]] \cong [B,[A,C]]$ satisfying S3.  In particular,  if $(\C,[-,-],I)$ is actually closed and admits such a symmetry then $(\C,\otimes ,I)$ is monoidal.
\item  If $(\C, [-,-],I,s)$ is symmetric closed then $(\C,\otimes,I,c)$ is symmetric monoidal.
\end{enumerate}
\end{Theorem}
\begin{proof}
For (1) we first prove that
\begin{equation}\label{eq:smc}
\cd{
[A \otimes B,[C,D]] \ar[d]_{t} \ar[rr]^{s} && [C,[A\otimes B,D]] \ar[d]^{[1,t]} \\ 
[A,[B,[C,D]]] \ar[r]_{[1,s]} & [A,[C,[B,D]]] \ar[r]_{s} & [C,[A,[B,D]]]}
\end{equation}
commutes.  From \eqref{eq:t} we have $t=[u,1] \circ L$.  So the upper path equals 
\begin{equation*}[
[1,[u,1]]  \circ [1,L] \circ s = [1,[u,1]] \circ s \circ [1,s] \circ L = s \circ [1,s] \circ [u,1] \circ L = s \circ [1,s] \circ t
\end{equation*}
by S3, naturality of $s$ twice applied, and the definition of $t$.\\
By assumption each component $v:\C(A,B) \to \C(I,[A,B])$ is invertible.  Accordingly a morphism $f:[A,B] \to [C,D]$ gives rise to a further morphism $f_{0}:\C(A,B) \to \C(C,D)$ by conjugating $\C(I,f)$.  At $[A,g]:[A,B] \to [A,C]$ we obtain $[A,g]_{0}=\C(A,g)$.  At $L:[A,B] \to [[C,A],[C,B]]$ an application of $C3$ establishes that $L_{0}=[C,-]:\C(A,B) \to \C([C,A],[C,B])$.  Now $(-)_{0}$, being defined by conjugating through natural isomorphisms, preserves composition.  Combining the two last cases we find that $t_{0} = [u,1]_{0} \circ L_{0}= \C(u,1) \circ [B,-] = \phi$, the adjointness isomorphism.  Applying $(-)_{0}$ to the above diagram, componentwise, then gives the commutative diagram below.
\begin{equation*}
\cd{
\C(A \otimes B,[C,D]) \ar[d]_{\phi} \ar[rr]^{s_{0}} && \C(C,[A\otimes B,D]) \ar[d]^{\C(C,t)} \\ 
\C(A,[B,[C,D]]) \ar[r]_{\C(A,s)} & \C(A,[C,[B,D]]) \ar[r]_{s_{0}} & \C(C,[A,[B,D]])}
\end{equation*}
Since the left vertical path and both horizontal paths are isomorphisms, so is $\C(C,t)$ for each $C$.  Therefore $t$ is itself an isomorphism.  The remainder of (1) now follows from Theorem~\ref{Theorem:EK}.\\
As mentioned, Part 2 follows from Proposition 2.3 of \cite{Day1978On}.  We note an alternative elementary argument.  Having established the commutativity of \eqref{eq:smc} and that $t$ is an isomorphism, we are essentially in the presence of what De Schipper calls a \emph{monoidal symmetric closed category}.\begin{footnote}{ This is not exactly the case as De Schipper, following Eilenberg-Kelly, includes a basic functor $V:\C \to \Set$ in his definition of symmetric closed category.  However this basic functor plays no role in the proof of the cited result.}\end{footnote}  Theorem 6.2 of \cite{deSchipper1975Symmetric} establishes that a monoidal symmetric closed category determines a symmetric monoidal one, as required.
\end{proof}

\section{From multicategories to skew closed categories}\label{section:multi}

Our examples of skew closed categories in Section 4 can be seen as arising from closed multicategories equipped with further structure.  In the present section we describe how to pass from such multicategories to skew closed categories.\\
Multicategories were introduced in \cite{Lambek1969Deductive2} and have objects $A,B,C \ldots $ together with \emph{multimaps} $(A_{1}, \ldots , A_{n}) \to B$ for each $n \in \mathbb N$.  These multimaps can be composed and satisfy natural associativity and unit laws.  We use boldface $\f C$ for a multicategory and $\C$ for its underlying category of unary maps.  A \emph{symmetric} multicategory $\f C$ comes equipped with actions of the symmetric group $S_{n}$ on the sets $\f C(A_{1}, \ldots A_{n};B)$ of $n$-ary multimaps.  These actions must be compatible with multimap composition.  For a readable reference on the basics of multicategories we refer to \cite{Leinster2004Higher}.

\subsection{Closed multicategories}
A multicategory $\f C$ is said to be \emph{closed} if for all $B, C \in \f C$ there exists an object $[B,C]$ and evaluation multimap $e:([B,C],B) \to C$
with the universal property that the induced function
\begin{equation}\label{eq:multi0}
{
\f C(A_{1}, \ldots, A_{n};[B,C]) \to \f C(A_{1}, \ldots, A_{n},B;C)}
\end{equation}
is a bijection for all $(A_{1}, \ldots , A_{n})$ and $n \in \mathbb N$.  We can depict the multimap $e:([B,C],B) \to C$ as below.
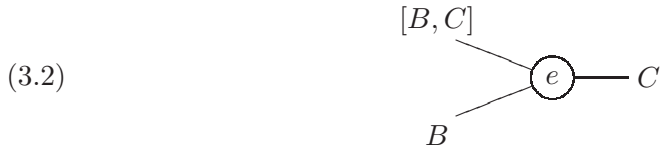
\begin{equation}\label{eq:multi0.5}
\xygraph{{e} *\xycircle<8pt>{-}="m" "m"(-[l(1)u(.4)] [l(.2)u(.2)] {[B,C]}, "m"(-[l(1)d(.4)] [l(.2)d(.2)] {B},-[r(1)] {C})} 
\end{equation}

\subsubsection{Semiclosed structure}
Using the above universal property one obtains a bifunctor $[-,-]:\C^{op} \times \C \to \C$.  Given $f:B \to C$ the map $[A,f]:[A,B] \to [A,C]$ is the unique one such that the two multimaps
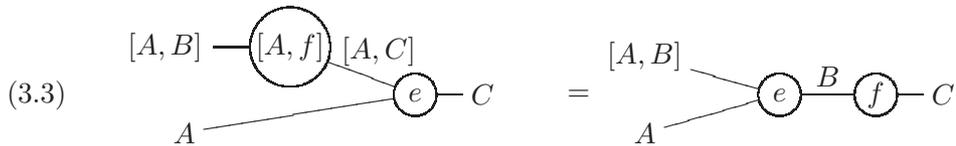
\begin{equation}\label{eq:multi1}
\xygraph{{e} *\xycircle<8pt>{-}="m" [l(1.3)u(0.5)] {[A,f]}*\xycircle<15pt>{-}="n" 
"n"(-"m"^-*{[A,C]}, -[l(0.8)] [l(0.5)] {[A,B]}
"m"(-[l(2.4)d(.4)] {A},-[r(0.7)] {C}  [r(1)] {=}) )} 
\xygraph{{e} *\xycircle<8pt>{-}="m" [r(1)] {f}*\xycircle<8pt>{-}="n" 
"n"(-[r(0.7)] {C}
"m"(-"n"^-*{B},-[l(1.4)d(.4)] {A}, -[l(1.4)u(.4)] {[A,B]})} 
\end{equation}
coincide, whilst $[f,A]$ is defined in a similar manner.\\
The natural bijections $$\f C([B,C],[A,B],A;C) \cong \f C([B,C],[A,B];[A,C]) \cong \f C([B,C],[[A,B],[A,C]])$$ induce a unique morphism $L:[B,C] \to [[A,B],[A,C]]$ such that the multimaps
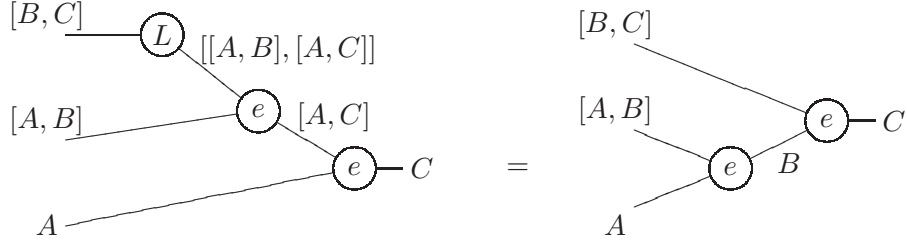
\begin{equation}\label{eq:multi2}
\xygraph{{e}*\xycircle<8pt>{-}="n" [u(0.6)l(1)] {e} *\xycircle<8pt>{-}="m" [u(0.8)l(1)] {L}*\xycircle<8pt>{-}="L" 
"L"(-"m"^-*{[[A,B],[A,C]]},-[l(1)] [u(.2)l(.2)] {[B,C]}
"m"(-"n"^-*{[A,C]},-[l(2)d(0.3)] [u(.2)l(.2)] {[A,B]}
"n"(-[l(3)d(0.6)] [l(.2)] {A}, -[r(0.5)] [r(.2)] {C}
 [r(1)] {=})
} 
\xygraph{{e}*\xycircle<8pt>{-}="n" [u(0.5)r(1)] {e} *\xycircle<8pt>{-}="m" 
"n"(-"m"_-*{B},-[l(1)u(.4)] [l(.2)u(.2)] {[A,B]}, -[l(1)d(.4)] [l(.2)d(.2)] {A}
"m"(-[l(2)u(.8)] [l(.2)u(.2)] {[B,C]}, -[r(0.5)] [r(.2)] {C}
)
}
\end{equation}
coincide.  
\subsubsection{Symmetry}
If $\f C$ is a symmetric multicategory the natural bijections 
$$\f C([A,[B,C]],B,A;C) \cong \f C([A,[B,C]],B;[A,C]) \cong \f C ([A,[B,C]],[B,[A,C]])$$ induce a unique map $s:[A,[B,C]] \to [B,[A,C]]$ such that the multimaps

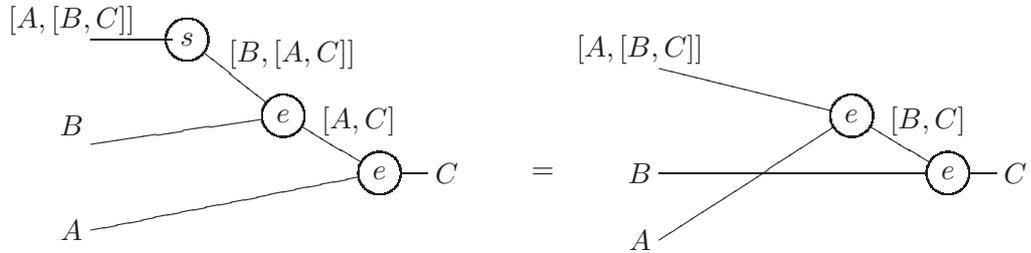
\begin{equation}\label{eq:multi3}
\xygraph{{e}*\xycircle<8pt>{-}="n" [u(0.6)l(1)] {e} *\xycircle<8pt>{-}="m" [u(0.8)l(1)] {s}*\xycircle<8pt>{-}="L" 
"s"(-"m"^-*{[B,[A,C]]},-[l(1)] [u(.2)l(.2)] {[A,[B,C]]}
"m"(-"n"^-*{[A,C]},-[l(2)d(0.3)] [u(.2)l(.2)] {B}
"n"(-[l(3)d(0.6)] [l(.2)] {A}, -[r(0.5)] [r(.2)] {C}
 [r(1)] {=})
} 
\xygraph{{e}*\xycircle<8pt>{-}="n" [u(0.6)l(1)] {e} *\xycircle<8pt>{-}="m"
"m"( -"n"^-*{[B,C]}, -[l(2)d(1.3)] [l(.2)] {A}, -[l(2)u(0.5)] [l(.2)u(0.2)] {[A,[B,C]]}
"n" (-[l(3)][l(0.2)]{B}, -[r(0.5)] [r(.2)] {C}))
}
\end{equation}
coincide.  The multimap above right depicts the image of $e \circ (e,1):([A,[B,C]],A,B) \to ([B,C],B) \to C$ under the action $$\f C([A,[B,C]],A,B;C) \cong \f C([A,[B,C]],B,A;C) \hspace{0.2cm}$$
of the symmetric group.  
\subsubsection{Nullary map classifiers and units}
\begin{Definition}
A multicategory $\f C$ has a nullary map classifier if there exists an object $I$ and multimap $u:(-) \to I$ such that the induced morphism $$\f C(u;A):\f C(I,A) \to \f C(-;A)$$
is a bijection for each $A$.
\end{Definition}
Equivalently, if the functor $\f C(-;?):\C \to \Set$ sending an object $A$ to the set of nullary maps $(-) \to A$ is representable.  In a closed multicategory a nullary map classifier $I$ enables the construction of morphisms $i:[I,A] \to A$ and $j:I \to [A,A]$.  The former is given by
\begin{equation}\label{eq:multiUnit}
\xygraph{{}*\xycircle<1pt>{-}="a" [r(1)] {u}*\xycircle<8pt>{-}="u" [u(.4)r(1.2)] {e} *\xycircle<8pt>{-}="m" 
"a"(-"u",
"u"(-"m"_-*{I},
"m"(-[l(3)u(.5)] [l(.2)u(.2)] {[I,A]}, "m"-[r(1)] {A})} 
\end{equation}
For $j$ observe that the identity $1:A \to A$ corresponds under the isomorphism $ \f C(A,A) \cong \f C(-;[A,A])$ to a nullary map $\hat{1}:(-) \to [A,A]$.  Now $j:I \to [A,A]$ is defined to be the unique map such that 
\begin{equation}\label{multiUnit2}
{
j \circ u = \hat{1} \hspace{0.2cm} .
}
\end{equation}
In \cite{Manzyuk2012Closed} the term \emph{unit} for a closed multicategory means something stronger than a nullary map classifier: it consists of a multimap $u:(-) \to I$ for which ~\eqref{eq:multiUnit} is invertible.  By Remark 4.2 of \emph{ibid.} a unit is a nullary map classifier.

\subsection{The results}The following result is Proposition 4.3 of \cite{Manzyuk2012Closed}.
\begin{Theorem}[Manzyuk]
If $\f C$ is a closed multicategory with unit $I$ then \newline $(\C,[-,-],I,L,i,j)$ is a closed category.
\end{Theorem}
Since we are interested in constructing mere skew closed categories a nullary map classifier suffices.  
\begin{Theorem}\label{thm:skewMulti}
If $\f C$ is a closed multicategory with a nullary map classifier $I$ then $(\C,[-,-],I,L,i,j)$ is a skew closed category.  Furthermore if $\f C$ is a symmetric multicategory then  $(\C,[-,-],I,L,i,j,s)$ is symmetric skew closed.
\end{Theorem}
\begin{proof}
We only outline the proof, which involves routine multicategorical diagram chases best accomplished using string diagrams as in \eqref{eq:multi0.5}--\eqref{eq:multiUnit}.  (We note that the deductions of C1 and C3 are given in the proof of Proposition 4.3 of \cite{Manzyuk2012Closed}.)  The axioms C1, C2 and C4 each assert the equality of two maps $$X \rightrightarrows [Y_{1}, \ldots ,[Y_{n-1},[Y_{n},Z]]..]$$ constructed using $[-,-]$, $L$, $i$ and $j$.  These correspond to the equality of the transposes $$(X, Y_{1}, \ldots, Y_{n}) \rightrightarrows Z$$ obtained by postcomposition with the evaluation multimaps.  Since $[-,-]$, $L$, $i$ and $j$ are defined in terms of their interaction with the evaluation multimaps only their definitions, together with the associativity and unit laws for a multicategory, are required to verify these axioms.  C3 and C5 each  concern the equality of two maps $I \rightrightarrows A$.  Here one shows that the corresponding nullary maps $(-) \rightrightarrows A$ coincide.  Again this is straightforward.  The axioms S1-S4 are verified in a similar fashion.
\end{proof}

Theorem~\ref{thm:skewMulti}, as stated, will not apply to the examples of interest, none of which quite has a nullary map classifier.  What we need is a generalisation that deals with combinations of strict and weak maps.
\begin{Definition}\label{Thm:strict}
Let $\f C$ be a multicategory equipped with a subcategory $\C_{s} \subseteq \C$ of \emph{strict} morphisms containing all of the identities.  We say that a multimap $f:(A_{1}, \ldots, A_{n}) \to B$ is \emph{strict in $i$} (or $A_{i}$ abusing notation) if for all families of multimaps $\{a_{j}:(-) \to A_{j}:j \in \{1,\ldots, i-1,i+1, \ldots, n\}\}$ the unary map $$f \circ (a_{1}, \ldots a_{i-1},1,a_{i+1}, \ldots a_{n}):A_{i} \to B$$ is strict.
\end{Definition}

\begin{Theorem}\label{Thm:multiToClosed}
Let $\f C$ be a closed multicategory equipped with a subcategory $\C_{s} \to \C$ of strict maps containing the identities.  Suppose further that
\begin{enumerate}
\item A multimap $(A_{1}, \ldots , A_{n},B) \to C$ is strict in $A_{i}$ if and only if its transpose $(A_{1}, \ldots, A_{n}) \to [B,C]$ is.
\item There is a multimap $u:(-) \to I$, precomposition with which induces a bijection $\f C(u,A):\C_{s}(I,A) \to \f C(-;A)$ for each $A$.
\end{enumerate}
Then $(\C,[-,-],L)$ is a semi-closed category.  Moreover $[-,-]$ and $L$ restrict to $\C_{s}$ where they form part of a skew closed structure $(\C_{s},[-,-],I,L,i,j)$.\\
Furthermore if $\f C$ is a symmetric multicategory then $(\C,[-,-],L,s)$ is symmetric semi-closed and $(\C_{s},[-,-],I,L,i,j,s)$ is symmetric skew closed.
\end{Theorem}
\begin{proof}
We must show that these assumptions ensure that the bifunctor $[-,-]:\C^{op} \times \C$ restricts to $\C_{s}$ and that the transformations $L, i, j$ and $s$ have strict components.  Beyond this point the proof is identical to that of Theorem~\ref{thm:skewMulti}.\\
A consequence of Definition~\ref{Thm:strict} is that multimaps strict in a variable are \emph{closed under composition}: that is, given $f:( A_{1}, A_{2} \ldots  A_{n}) \to  B_{k}$ strict in $ A_{i}$ and $g:( B_{1}, B_{2} \ldots  B_{m}) \to  C$ strict in $ B_{k}$  the composite multimap $$( B_{1}\ldots  B_{k-1}, A_{1}\ldots  A_{i} \ldots  A_{n}, B_{k+1} \ldots  B_{m}) \to  C$$ is strict in $ A_{i}$.  We use this fact freely in what follows.\\
Observe that since $1:[A,B] \to [A,B]$ is strict its transpose, the evaluation multimap $e:([A,B],A) \to B$, is strict in $[A,B]$.  It follows that if $f:B \to C$ is strict then the composite multimap $f \circ e: ([A,B],A) \to B \to C$ of \eqref{eq:multi1} is strict in $[A,B]$.  Accordingly its transpose $[A,f]:[A,B] \to [A,C]$ is strict.  Likewise $[f,A]$ is strict if $f$ is.  Hence $[-,-]$ restricts to $\C_{s}$.\\
Since evaluation multimaps are strict in the first variable the composite $e \circ (1,e):([B,C],[A,B],A) \to C$ of \eqref{eq:multi2} is strict in $[B,C]$.  Since transposing this twice yields $L:[B,C] \to [[A,B],[A,C]]$ we conclude that $L$ is strict.  The composite $i = e \circ (1,u):[I,A] \to A$ of \eqref{eq:multiUnit} is strict as $e:([I,A],I) \to A$ is strict in the first variable.  Clearly $j$ is strict.\\
In a symmetric multicategory the actions of the symmetric group commute with composition.  It follows that if  $F:(A_{1}, \ldots A_{n}) \to B$ is strict in $A_{i}$ and $\phi \in Sym(n)$ then $\phi(F):(A_{\phi(1)}, \ldots, A_{\phi(n)}) \to B$ is strict in $A_{\phi(i)}$.  Therefore the composite $([A,[B,C]],B,A) \to C$ on the right hand side of \eqref{eq:multi3} is strict in $[A,[B,C]]$.  Since $s:[A,[B,C]] \cong [B,[A,C]]$ is obtained by transposing this twice, it follows that $s$ is strict.
\end{proof}

In the next section we will encounter several examples of $\Cat$-enriched multicategories, hence 2-multicategories \cite{Hyland2002Pseudo}.  A 2-multicategory $\f C$ has \emph{categories} $\f C(A_{1}, \ldots, A_{n};B)$ of multilinear maps and transformations between, and an extension of multicategorical composition dealing with these transformations.  There is an evident notion of \emph{closed 2-multicategory}, in which the bijection \eqref{eq:multi0} is replaced by an isomorphism, and of symmetric 2-multicategory.  
Theorem~\ref{Thm:multiToClosed} generalises straightforwardly to 2-multicategories as we now record. 

\begin{Theorem}\label{Thm:multiToClosed2}
Let $\f C$ be a closed 2-multicategory equipped with a locally full sub 2-category $\C_{s} \to \C$ of strict maps containing the identities.  Suppose further that
\begin{enumerate}
\item A multimap $(A_{1}, \ldots , A_{n},B) \to C$ is strict in $A_{i}$ if and only if its transpose $(A_{1}, \ldots, A_{n}) \to [B,C]$ is.
\item There is a multimap $u:(-) \to I$, precomposition with which induces an isomorphism $\f C(u,A):\C_{s}(I,A) \to \f C(-;A)$ for each $A$.
\end{enumerate}
Then $(\C,[-,-],L)$ is a semi-closed 2-category.  Moreover $[-,-]$ and $L$ restrict to $\C_{s}$ where they form part of a skew closed 2-category $(\C_{s},[-,-],I,L,i,j)$.\\
Furthermore if $\f C$ is a symmetric 2-multicategory then $(\C,[-,-],L,s)$ is symmetric semi-closed and $(\C_{s},[-,-],I,L,i,j,s)$ is a symmetric skew closed 2-category.
\end{Theorem}

\begin{Remark}
Theorem 5.1 of \cite{Manzyuk2012Closed} shows that the notions of closed multicategory with unit and closed category are, in a precise sense, equivalent.  We do not know whether skew closed categories are equivalent to some kind of multicategorical structure.
\end{Remark}

\section{Examples of skew closed structures}\label{section:examples}
The goal of this section is to describe a few concrete examples of the kind of skew closed structures that we are interested in.  All can be seen to arise from multicategories although sometimes it will be easier to describe the skew closed structure directly.  \\
In each case we meet a category, or 2-category, $\C$ of weak maps equipped with a subcategory $\C_{s}$ of strict maps.  The subcategory of strict maps is well behaved -- locally presentable, for instance -- whereas $\C$ is not.  The objects of the internal hom $[A,B]$ are the \emph{weak maps} but these only form part of a skew closed structure on the subcategory $\C_{s}$ of \emph{strict maps}.

\subsection{Categories with structure}\label{section:structuredCats}
The following examples can be understood as arising from pseudo-commutative 2-monads in the sense of \cite{Hyland2002Pseudo} -- this more abstract approach is described in Section~\ref{section:pseudoCommutative}.
\subsubsection{Categories with specified limits}\label{section:limits}
Let $\mathbf{D}$ be a set of small categories, thought of as diagram types.  There is a symmetric 2-multicategory $\DLIM$ whose objects $\f A$ are categories $A$ equipped with a choice of $\f D$-limits.  The objects of the category $\DLIM(\f A_{1}, \ldots, \f A_{n};\f B)$ are functors $F:A_{1} \times \ldots \times A_{n} \to B$ preserving $\mathbf{D}$-limits in each variable, and the morphisms are just natural transformations.  For the case $n=0$ we have $\DLIM(-;\f B)=B$.\\
The morphisms in the 2-category of unary maps $\DLim$ are the $\mathbf{D}$-limit preserving functors, amongst which we have the 2-category $\DLim_{s}$ of strict $\mathbf{D}$-limit preserving functors and the inclusion $j:\DLim_{s} \to \DLim$.  Accordingly a multimap $\f{F}:(\f A_{1},\ldots, \f A_{n}) \to \f{B}$ is strict in $\f{A_{i}}$ just when each functor $F(a_{1}, \ldots ,a_{i-1},-,a_{i+1}, \ldots, a_{n}):A_{i} \to B$ preserves $\mathbf{D}$-limits strictly.\\
The functor category $[A,B]$ has a canonical choice of $\mathbf{D}$-limits inherited pointwise from $\f B$.  Since \emph{$\mathbf{D}$-limits commute with $\mathbf{D}$-limits} the full subcategory $j:\DLim(\f A,\f B) \to [A,B]$ is closed under their formation (although $\DLim_{s}(\f A,\f B)$ is not!) and we write $[\f A,\f B]$ for $\DLim(A,B)$ equipped with this choice of $\f D$-limits.\\
It is routine to verify that the objects $[\f A,\f B]$ exhibit $\DLim$ as a closed 2-multicategory and moreover that a multimap $\f{F}:(\f{A}_{1},\ldots, \f{A}_{n},\f{B}) \to \f{C}$ is strict in $\f A_{i}$ just when its transpose  $(\f{A}_{1},\ldots, \f{A}_{n}) \to [\f{B},\f{C}]$ is.  With regards units, the key point is that the forgetful 2-functor $U:\DLim_{s} \to \Cat$ has a left 2-adjoint $F$.  This follows from \cite{Blackwell1989Two-dimensional} but see also Section~\ref{section:2-monads}.  Accordingly we have a natural isomorphism
$\DLIM(-;\f A) \cong \Cat(1,A) \cong \DLim_{s}(F1,\f A)$.  By Theorem~\ref{Thm:multiToClosed2} we obtain the structure of a symmetric semi-closed 2-category $(\DLim,[-,-],L,s)$ restricting to a symmetric skew closed 2-category $(\DLim,[-,-],F1,L,i,j,s)$.\\
The skew closed structure on $\DLim_{s}$ fails to extend to $\DLim$ because the unit map $j:F1 \to [\f A,\f A]$ is only \emph{pseudo-natural} in morphisms of $\DLim$. (It does, however, extend to a \emph{pseudo-closed} structure on $\DLim$ in the sense of \cite{Hyland2002Pseudo}).  The skew closed $\DLim_{s}$ is neither left nor right normal: for example, the canonical functor $\DLim_{s}(\f A,\f B) \to \DLim_{s}(F1,[\f A,\f B])$ is isomorphic to the inclusion $\DLim_{s}(\f A,\f B) \to \DLim(\f A,\f B)$ and this is not in general invertible.

\subsubsection{Permutative categories and so on}\label{section:permutative}
An example amenable to calculation concerns symmetric strict monoidal -- or permutative -- categories.  The symmetric skew closed structure can be seen as arising from a symmetric 2-multicategory, described in \cite{Elmendorf2006Rings}.  Because the relevant definition of multilinear map is rather long, we treat the skew closed structure directly.\\
Let $\Perm_{s}$ and $\Perm$ denote the 2-categories of permutative categories with the strict symmetric monoidal and strong symmetric monoidal functors between.  Using the symmetry of $\f B$ the category $\Perm(\f A,\f B)$ inherits a pointwise structure $[\f A,\f B] \in \Perm$.  Namely we set $(F \otimes G)-= F- \otimes G-$.  The structure isomorphism $(F \otimes G)(a \otimes b) \cong (F\otimes G)a \otimes (F\otimes G)b$ combines the structure isomorphisms $F(a \otimes b) \cong Fa \otimes Fb$ and $G(a \otimes b) \cong Ga \otimes Gb$ with the symmetry as below:
$$F(a \otimes b) \otimes G(a \otimes b) \cong Fa \otimes Fb \otimes Ga \otimes Gb \cong Fa \otimes Ga \otimes Fb \otimes Gb \hspace{0.5cm} .$$
The structural isomorphism concerning monoidal units is obvious.  The hom objects $[A,B]$ extend in the obvious way to a 2-functor $[-,-]:\Perm^{op} \times \Perm \to \Perm$.  Moreover, the functor $[\f C,-]:\Perm(\f A, \f B) \to \Perm([\f C,\f A], [\f C,\f B])$  lifts to a strict map $L=[\f C,-]:[\f A,\f B] \to [[\f C,\f A],[\f C,\f B]]$ because both domain and codomain have structure inherited \emph{pointwise from $\f B$}.  
We omits details of the symmetry isomorphism $s:[\f A,[\f B,\f C]] \cong [\f B,[\f A,\f C]]$.\\
The unit $F1$ is the free permutative category on $1$: the category of finite ordinals and bijections.  The unit map $i:[F1,\f A] \to \f A$ is given by evaluation at $1$ whilst $j:F1 \to [\f A,\f A]$ is the unique symmetric strict monoidal functor with $j(1)=1_{\f A}$.  Again the skew closed structure is neither left nor right normal.\\
This example can be generalised to deal with general symmetric monoidal categories.  A careful analysis of both tensor products and internal homs on the 2-category $\SMon$ of symmetric monoidal categories and strong symmetric monoidal functors was given by Schmitt \cite{Schmitt2007Tensor}.

\subsection{2-categories and bicategories}\label{section:bicategories}
Examples of skew closed structures not arising from pseudo-commutative 2-monads, even in the extended sense of \cite{Lopez-Franco2011Pseudo}, include 2-categories and bicategories.  We focus upon the more complex case of bicategories.
Let $\Bicat$ denote the category of bicategories and homomorphisms (also called pseudofunctors) and $\Bicat_{s}$ the subcategory of bicategories and strict homomorphisms.  We describe a symmetric skew closed structure on $\Bicat_{s}$ with internal hom $Hom(A,B)$ the bicategory of pseudofunctors, pseudonatural transformations and modifications from $A$ to $B$.\\
This skew closed structure arises from a closed symmetric multicategory.  We begin by briefly recalling the multicategory structure, which was introduced and studied in depth in Section 1.3 of \cite{Verity1992Enriched} by Verity, and to which we refer for further details.  The multicategory $\BICAT$ -- denoted by $\underline{Hom_{s}}$ in \emph{ibid.} -- has bicategories as objects.  The multimaps are a variant of the cubical functors of \cite{Gray1974Formal}.  More precisely, a multimap $F:(A_{1}, \ldots ,A_{n}) \to B$ consists of
\begin{itemize}
\item for each $n$-tuple $(a_{1}, \ldots , a_{n})$ an object $F(a_{1}, \ldots ,a_{n})$ of $B$;
\item for each $1 \leq i \leq n$ a homomorphism $F(a_{1}, \ldots , a_{i-1},-,a_{i+1} \ldots, a_{n})$ extending the above function on objects;
\item for each pair $1 \leq i < j \leq n$, $n$-tuple of objects $(a_{1}, \ldots , a_{n})$ and morphisms $f:a_{i} \to a^{\prime}_{i} \in A_{i}$ and  $f_{j}:a_{j} \to a^{\prime}_{j} \in A_{j}$, an invertible 2-cell:
\begin{equation*} 
\xy
(0,0)*+{F( a_{i}, a_{j} )}="00"; (45,0)*+{F( a^{\prime}_{i}, a_{j} )}="10"; (0,-15)*+{F( a_{i}, a^{\prime}_{j} )}="01"; (45,-15)*+{F( a^{\prime}_{i}, a^{\prime}_{j} )}="11"; 
{\ar^{F( f_{i}, a_{j} )} "00"; "10"}; {\ar_{F( f_{i}, a^{\prime}_{j} )} "01"; "11"}; 
{\ar_{F( a_{i}, f_{j} )} "00"; "01"}; {\ar^{F( a^{\prime}_{i}, f_{j} )} "10"; "11"}; 
{\ar@{=>}^{F( f_{i}, f_{j} )}(22,-6)*+{};(16,-10)*+{}};
\endxy
\end{equation*}
where we have omitted to label the inactive parts of the $n$-tuple $(a_{1}, \ldots , a_{n})$ under the action of $F$.  These invertible 2-cells are required to form the components of pseudonatural transformations both vertically -- $F(-_{i},f_{j}):F(-_{i},a_{j}) \to F(-_{i},a^{\prime}_{j})$ -- and horizontally -- $F(f_{i},-_{j}):F(a_{i},-_{j}) \to F(a^{\prime}_{i},-_{j})$ -- and satisfy a further cubical identity involving trios of morphisms.
\end{itemize}
A nullary morphism $(-) \to B$ is simply defined to be an object of $B$.  Observe that the category of unary maps of $\BICAT$ is simply $\Bicat$.
It is established in Section 1.3 of \emph{ibid.} -- see Lemma 1.3.4 and the discussion that follows -- that the symmetric multicategory $\BICAT$ is closed, with hom-object given by the bicategory $Hom(A,B)$ of homomorphisms, pseudonatural transformations and modifications from $A$ to $B$.\\
A multimap $F:(A_{1}, \ldots ,A_{n}) \to B$ is strict in $A_{i}$ just when each homomorphism $F(a_{1}, \ldots , a_{i-1},-,a_{i+1} \ldots, a_{n})$ is strict.  An inspection of the bijection of Lemma 1.3.4 of \emph{ibid.} makes it clear that the natural bijection $$\BICAT(A_{1}, \ldots, A_{n},B;C) \cong \BICAT(A_{1}, \ldots, A_{n};Hom(B,C))$$ respects strictness in $A_{i}$.\\
Turning to the unit, recall that $\BICAT(-;A)=A_{0}$.   The forgetful functor $(-)_{0}:\Bicat_{s} \to \Set$ has a left adjoint $F$ for general reasons -- see Section~\ref{section:homotopicalBicats} for more on this. It follows that we have a bijection $\Bicat_{s}(F1,A) \cong \BICAT(-;A)$ where $F1$ is the free bicategory on $1$.  Concretely, $F1$ has a single object $\bullet$ and a single generating 1-cell $e:\bullet \to \bullet$.  General morphisms are (non-empty) bracketed copies of $e$ such as $((ee)e)$, and two such morphisms are connected by a unique 2-cell, necessarily invertible.\\
By Theorem~\ref{Thm:multiToClosed} we obtain a symmetric semi-closed category $(\Bicat_{s},Hom,L)$ which restricts to a symmetric skew closed structure $(\Bicat_{s},Hom,F1,L,i,j)$.  As in the preceding examples the skew closed structure on $\Bicat_{s}$ is not closed and fails to extend to $\Bicat$.\\
In Section~\ref{section:homotopicalBicats} we further analyse this symmetric skew closed structure.  Accordingly we describe a few aspects of it in more detail.  
Firstly, let us describe the action of the functor $Hom(-,-):\Bicat^{op} \times \Bicat \to \Bicat$.  From a homomorphism $f:A \rightsquigarrow B$ the homomorphism $Hom(f,1):Hom(B,C) \to Hom(A,C)$ obtained by precomposition is always strict, and straightforward to describe.  The postcomposition map $Hom(1,f)=Hom(B,C) \to Hom(B,D)$ induced by a strict homomorphism $f:C \to D$ is equally straightforward.\\
Though not strictly required in what follows, for completness we mention the slightly more complex case where $f$ is non-strict. At $\eta:g \to h \in Hom(B,C)$ the pseudonatural transformation $f\eta:fg \to fh$ has components $f\eta_{a}:fga \to fha$ at $a \in B$; at $\alpha:a \to b$ the invertible 2-cell $(f\eta)_{\alpha}$:
 \begin{equation*}
\cd{
 fh\alpha \circ f\eta_{a} \ar@{=>}[r]^{\lambda_{f}} & f(h{\alpha} \circ \eta_{a}) \ar@{=>}[r]^{f\eta_{\alpha}} & f(\eta_{b} \circ g\alpha) \ar@{=>}[r]^{{\lambda_{f}}^{-1}} & f\eta_{b} \circ fg\alpha}
 \end{equation*}
 conjugates $f\eta_{\alpha}$ by the coherence constraints for $f$.  The action of $f_{*}$ on 2-cells is straightforward.  The coherence constraints $f(\eta \circ \mu) \cong f(\eta) \circ f(\mu)$ and $f(id_{g}) \cong id(fg)$ for $f^{*}$ are pointwise those for $f$.\\
The only knowledge required of $L:Hom(B,C) \to Hom(Hom(A,B),Hom(A,C))$ is that it has underlying function $$Hom(A,-):\Bicat(B,C) \to \Bicat(Hom(A,B),Hom(A,C)) \hspace{0.5cm}.$$  The unit map $i:Hom(F1,A) \to A$ evaluates at the single object $\bullet$ of $F1$ whilst $j:F1 \to Hom(A,A)$ is the unique strict homomorphism sending $\bullet$ to the identity on $A$.\\
This example can be modified to deal with 2-categories.  Let $\TWOCAT \subset \BICAT$ and $\twocat_{s} \subset \Bicat_{s}$ be the symmetric multicategory and category obtained by restricting the objects from bicategories to 2-categories.  Since $Hom(A,B)$ is a 2-category if $B$ is, we obtain a closed multicategory $\TWOCAT$ by restriction.  In this case we have a natural bijection $\twocat_{s}(1,A)\cong A_{0} = \TWOCAT(-;A)$. It follows that we obtain a symmetric skew closed structure $(\twocat_{s},Hom,L,1,i,j)$ with the same semi-closed structure as before, but with the simpler unit $1$.

\subsection{Lax morphisms}
Each of the above examples describes a symmetric skew closed structure arising from a symmetric closed multicategory.  In each case there are non-symmetric variants dealing with lax structures, of which we mention a few now.  These have the same units but different internal homs. In $\DLim_{s}$ the hom $[\f A,\f B]$ is the functor category $[A,B]$ equipped with $\f D$-limits pointwise in $B$.  In $\Perm_{s}$ the internal hom $[\f A,\f B]$ consists of lax monoidal functors and monoidal transformations.  For $\Bicat_{s}$ one can take $[A,B]$ to be the bicategory of homomorphisms and lax natural transformations from $A$ to $B$.

\section{Skew structures descending to the homotopy category}
In the present section we consider categories $\C$ equipped with a Quillen model structure as well as a skew monoidal or skew closed structure. 
We describe conditions under which the skew structures descend to the homotopy category $Ho(\C)$ and call the skew monoidal/closed structures on $\C$ \emph{homotopy monoidal/closed} if the induced structures on $Ho(\C)$ are \emph{genuinely monoidal/closed}.   Theorem~\ref{Theorem:TotalDerived} gives a complete description of how monoidal skew closed structure descends to the homotopy category.  Our analogue of Eilenberg and Kelly's theorem is Theorem~\ref{Theorem:hEK}: it allows us to recognise homotopy monoidal structure in terms of homotopy closed structure.\\
We assume some familiarity with the basics of Quillen model categories, as introduced in \cite{Quillen1967Homotopical}, and covered in Chapter 1 of \cite{Hovey1999Model}.  Let us fix some terminology and starting assumptions.  We assume that all model categories $\C$ have functorial factorisations.  It follows that $\C$ is equipped with cofibrant and fibrant replacement functors $Q$ and $R$ together with natural transformations $p:Q \to 1$ and $q:1 \to R$  whose components are respectively trivial fibrations and trivial cofibrations.  Let $j:\C_{c} \to \C$ and $j:\C_{f} \to \C$ denote the full subcategories of cofibrant and fibrant objects, through which $Q$ and $R$ respectively factor.  The four functors preserve weak equivalences and hence extend to the homotopy category.  At that level we obtain adjoint equivalences
\begin{equation*}
\cd{Ho(\C_{c}) \ar@/^1ex/[r]^{Ho(j)} & Ho(\C) \ar@/^{1ex}/[l]^{Ho(Q)}  && Ho(\C_{f}) \ar@/^1ex/[r]^{Ho(j)} & Ho(\C) \ar@/^{1ex}/[l]^{Ho(R)}}
\end{equation*}
with counit and unit given by $Ho(p)$ and its inverse, and $Ho(q)$ and its inverse respectively.  If a functor between model categories $F:\C \to \D$ preserves weak equivalences between cofibrant objects we can form its \emph{left derived functor} $F_{l}=Ho(FQ):Ho(\C) \to Ho(\D)$, equally $Ho(Fj)Ho(Q):Ho(\C) \to Ho(\C_{c}) \to Ho(\D)$.  If $G$ preserves weak equivalences between fibrant objects then $G_{r}=Ho(GR)=Ho(Gj)Ho(R)$ is its \emph{right derived functor.}
\subsection{Skew monoidal structure on the homotopy category}\label{section:hmonoidal}
Let $\C$ be a model category equipped with a skew monoidal structure $(\C,\otimes,I,\alpha,l,r)$.  Our interest is in left deriving this to a skew monoidal structure on $Ho(\C)$.  In the monoidal setting this was done in \cite{Hovey1999Model} and the construction in the skew setting, described below, is essentially identical.
\begin{AxiomM}
$\otimes:\C \times \C \to \C$ preserves cofibrant objects and weak equivalences between them and the unit $I$ is cofibrant.
\end{AxiomM}
The above assumption ensures that the skew monoidal structure on $\C$ restricts to one on $\C_{c}$ and that the restricted functor $\otimes:\C_{c} \times \C_{c} \to \C_{c}$ preserves weak equivalences.  Accordingly we obtain a skew monoidal structure $(Ho(\C_{c}),Ho(\otimes),I)$ with the same components as before. Transporting this along the adjoint equivalence $Ho(j):Ho(\C_{c}) \leftrightarrows Ho(\C):Ho(Q)$ yields a skew monoidal structure $$(Ho(\C),\otimes_{l},I,\alpha_{l},l_{l},r_{l})$$ on $Ho(\C)$.  We will often refer to $(Ho(\C),\otimes_{l},I)$ as the \emph{left-derived} skew monoidal structure since $\otimes_{l}$ is the left derived functor of $\otimes$.  On objects we have $A \otimes_{l} B=QA \otimes QB$ and an easy calculation shows that the constraints for the skew monoidal structure are given by the following maps in $Ho(\C)$.
\begin{equation}\label{eq:LD1}
\cd{Q(QA \otimes QB) \otimes QC \ar[r]^{\h{p \otimes 1}} & (QA \otimes QB) \otimes QC \ar[d]^{\h{\alpha}} \\ & QA \otimes (QB \otimes QC) \ar[r]^{\h{(1 \otimes p})^{-1}} & QA \otimes Q(QB \otimes QC)}
\end{equation}
\begin{equation}\label{eq:LD2}
\cd{QI \otimes QA \ar[rr]^{\h{p \otimes 1}}&&  I \otimes QA \ar[rr]^{\h{l}} && QA \ar[rr]^{\h{p}} && A}
\end{equation}
\begin{equation}\label{eq:LD3}
\cd{A \ar[rr]^{\h{p}^{-1}} && QA \ar[rr]^{r} & & {QA \otimes I} \ar[rr]^{(\h{1 \otimes p})^{-1}}  && QA \otimes QI}
\end{equation}

\begin{Definition}
Let $(\C,\otimes,I)$ be a skew monoidal structure on a model category $\C$ satisfying Axiom M.  We say that $\C$ is \emph{homotopy monoidal} if $(Ho(\C),\otimes_{l},I)$ is genuinely monoidal.
\end{Definition}
\begin{Proposition}\label{prop:hmonoidal}
Let $(\C,\otimes,I)$ be a skew monoidal category with a model structure satisfying Axiom M. The following are equivalent.
\begin{enumerate}
\item
$(\C,\otimes,I)$ is homotopy monoidal.
\item For all cofibrant $X,Y,Z$ the map $\alpha:(X \otimes Y) \otimes Z \to X \otimes (Y \otimes Z)$ is a weak equivalence, and for all cofibrant $X$ both maps $r:X \to X \otimes I$ and $l:I \otimes X \to X$ are weak equivalences.
\end{enumerate}
\end{Proposition}
\begin{proof}
Observe that the constraints~\eqref{eq:LD1},\eqref{eq:LD2} and \eqref{eq:LD3} are $\alpha_{QA,QB,QC}$, $l_{QA}$ and $r_{QA}$ conjugated by isomorphisms in $Ho(\C)$.  It follows that $(Ho(\C),\otimes_{l},I)$ is genuine monoidal just when for $\alpha_{QA,QB,QC}$, $l_{QA}$ and $r_{QA}$ are isomorphisms in $Ho(\C)$ for all $A,B$ and $C$.  Axiom M ensures for that cofibrant $A,B,C$ that we have isomorphisms $\alpha_{A,B,C} \cong \alpha_{QA,QB,QC}$, $l_{A} \cong l_{QA}$ and $r_{A} \cong r_{QA}$ in $Ho(\C)^{\atwo}$ so that the former maps are isomorphisms just when the latter ones are.  This proves the claim.
\end{proof}
\begin{Not}
We call $(\C,\otimes,I)$ \emph{homotopy symmetric monoidal} if $(Ho(\C),\otimes_{l},I)$ admits the further structure of a symmetric monoidal category, but emphasise that this refers to a symmetry on $Ho(\C)$ not necessarily arising from a symmetry on $\C$ itself.
\end{Not}
\subsection{Skew closed structure on the homotopy category}\label{section:hclosed}
Let $\C$ be a model category equipped with a skew closed structure  $(\C,[-,-],I,L,i,j)$. Our intention is to \emph{right derive} the skew closed structure to the homotopy category.  This construction, more complex than its monoidal counterpart, is closely related to the construction of a skew closed category $(\C,[Q-,-],I)$ from a closed comonad $Q$ \cite{Street2013Skew}.

\begin{AxiomC}
For cofibrant $X$ the functor $[X,-]$ preserves fibrant objects and trivial fibrations. For fibrant $Y$ the functor $[-,Y]$ preserves weak equivalences between cofibrant objects.  The unit $I$ is cofibrant.\begin{footnote}{ We could weaken Axiom C by requiring that $[X,-]$ preserves only weak equivalences between fibrant objects, rather than all trivial fibrations.  This is still enough to construct the skew closed structure of Theorem~\ref{Theorem:hclosed1} though the proof becomes slightly longer.  Because we need the stronger Axiom MC in the crucial monoidal skew closed case anyway, we emphasise the convenient Axiom C.}\end{footnote}
\end{AxiomC}
It follows from Axiom C that if $A$ is cofibrant then $[1,p_{B}]:[A,QB] \to [A,B]$ is a trival fibration.  Accordingly we obtain a lifting $k_{A,B}$ as below.
\begin{equation}\label{eq:k0}
\cd{Q[A,B] \ar[dr]_{p_{[A,B]}} \ar[rr]^{k_{A,B}} && [A,QB] \ar[dl]^{[1,p_{B}]} \\
& [A,B]}
\end{equation}
Because $I$ is cofibrant we also have a lifting $e$ as below.
\begin{equation}\label{eq:e}
\cd{I \ar[dr]_{1} \ar[rr]^{e} && QI \ar[dl]^{p_{I}} \\
& I}
\end{equation}
\begin{Lemma}\label{Theorem:LemmaK}
Let $(\C,[-,-],I)$ satisfy Axiom C.  Then each of the following four diagrams
\begin{equation}\label{eq:k1}
\def\objectstyle{\scriptstyle}
\def\labelstyle{\scriptstyle}
\cd{Q[QB,C] \ar[d]_{k} \ar[r]^<<<<{QL} & Q[[QA,QB],[QA,C]] \ar[r]^{Q[k,1]} & Q[Q[QA,B],[QA,C]] \ar[r]^{k} & [Q[QA,B],Q[QA,C]] \ar[d] ^{[1,k]} \\
[QB,QC] \ar[r]_<<<<{L} & [[QA,QB],[QA,QC]] \ar[rr]_{[k,1]} && [Q[QA,B],[QA,QC]]} 
\end{equation}
\begin{equation}\label{eq:k2}
\cd{Q[I,B] \ar[dr]_{Qi} \ar[rr]^{k} && [I,QB] \ar[dl]^{i} \\
& QB}
\end{equation}
\begin{equation}\label{eq:k3}
\cd{I \ar@/_2pc/[rrrr]_{j} \ar[r]^{e} & QI \ar[r]^{Qj} & Q[A,A] \ar[r]^{Q[p,1]} & Q[QA,A] \ar[r]^{k} & [QA,QA]}
\end{equation}
\begin{equation}\label{eq:k4}
\cd{Q[B,C] \ar[d]_{Q[f,g]} \ar[r]^{k_{B,C}} & [B,QC] \ar[d]^{[f,Qg]} \\
Q[A,D] \ar[r]_{k_{A.D}} & [A,QD]}
\end{equation}
commutes up to left homotopy.   Moreover, if $X$ is fibrant then the image of each diagram under $[-,X]$ commutes in $Ho(\C)$.
\end{Lemma}
Note that in ~\eqref{eq:k4} $A$ and $B$ are cofibrant and the morphisms $f:B \to A$ and $g:C \to D$ are arbitrary.
\begin{proof}
In each case we are presented with a pair of maps $f,g:U \rightrightarrows V$ with $U$ cofibrant.  To prove that $f$ and $g$ are left homotopic it suffices, by Proposition 1.2.5(iv) of \cite{Hovey1999Model}, to show that there exists a trivial fibration $h:V \to W$ with $h\circ f = h \circ g$. We take the trivial fibrations $[1,p_{[QA,QC]}]$, $p_{B}$, $[1,p_{A}]$ and $[1,p_{D}]$ respectively.  Each diagram, postcomposed with the relevant trivial fibration, is easily seen to commute.\\
For the second point observe that any functor $\C \to \D$ sending weak equivalences between cofibrant objects to isomorphisms identifies left homotopic maps - this follows the proof of Corollary 1.2.9 of \emph{ibid}.  Applying this to the composite of $[-,X]:\C \to \C^{op}$ and $\C^{op} \to Ho(\C)^{op}$ gives the result.
\end{proof}
Axiom C ensures that the right derived functor $$[-,-]_{r}:Ho(\C)^{op} \times Ho(\C) \to Ho(\C)$$ exists with value $[A,B]_{r} = [QA,RB]$.  The unit for the skew closed structure will be $I$.  Using Axiom C we form transformations $L_{r}$, $i_{r}$ and $j_{r}$ on $Ho(\C)$ as below.
\begin{equation}\label{eq:RD1}
\cd{[QA,RB]  \ar[d]_{[Qq,1]^{-1}} & & [Q[QC,RA],R[QC,RB]]\\
 [QRA,RB] \ar[r]^<<<<{L} &  [[QC,QRA],[QC,RB]] \ar[r]^{[k,1]} & [Q[QC,RA],[QC,RB]] \ar[u]_{[1,q]}}
\end{equation}
\begin{equation}\label{eq:RD2}
\cd{[QI,RA] \ar[r]^{[e,1]} & [I,RA] \ar[r]^{i} & RA \ar[r]^{q^{-1}} & A}
\end{equation}
\begin{equation}\label{eq:RD3}
\cd{I \ar[r]^{j} & [A,A] \ar[r]^{[p,1]} &[QA,A] \ar[r]^{[1,q]} & [QA,RA]}
\end{equation}
\begin{Theorem}\label{Theorem:hclosed1}
Let $\C$ be a model category equipped with a skew closed structure $(\C,[-,-],I)$ satisfying Axiom C. Then $Ho(\C)$ admits a skew closed structure $(Ho(\C),[-,-]_{r},I)$ with constraints as above.
\end{Theorem}
\begin{proof}
In order to keep the calculations relatively short we will first describe a slightly simpler skew closed structure on $Ho(\C_{f})$.  We then obtain the skew closed structure on $Ho(\C)$ by transport of structure.\\
So our main task is to construct a suitable skew closed structure on $Ho(\C_{f})$. Now Axiom C ensures that $[QA,B]$ is fibrant whenever $B$ is.  The restricted bifunctor $[Q-,-]:\C_{f}^{op} \times \C_{f} \to \C_{f}$ then preserves weak equivalences in each variable and so extends to a bifunctor $Ho([Q-,-])$ on $Ho(\C_{f})$.  For the unit on $Ho(\C_{f})$ we take $RI$.\\
The constraints are given by the following three maps.
\begin{equation}\label{eq:RD1}
\cd{[QA,B] \ar[r]^<<<<{L} & [[QC,QA],[QC,B]] \ar[r]^{[k,1]} & [Q[QC,A],[QC,B]]}
\end{equation}
\begin{equation}\label{eq:RD2}
\cd{[QRI,A] \ar[r]^{[Qq,1]} & [QI,A] \ar[r]^{[e,1]} & [I,A] \ar[r]^{i} & A}
\end{equation}
\begin{equation}\label{eq:RD3}
\cd{RI \ar[r]^{q^{-1}} & I \ar[r]^{j} & [A,A] \ar[r]^{[p,1]} & [QA,A]}
\end{equation}
We should explain why the above components are natural on $Ho(\C_{f})$ -- in the appropriate variance -- since consideration of extraordinary naturality is perhaps non-standard.\\
Given $F,G:\A \rightrightarrows \B$ and a family of maps $\{\eta_{A}:FA \to GA: A \in \A \}$ we can consider the class of morphisms $Nat(\eta) \subseteq Mor(\A)$ with respect to which $\eta$ is natural.  $Nat(\eta)$ is closed under composition and inverses in $\A$.   If $(\A,\W)$ is a category equipped with a collection of weak equivalences $\W$ then each arrow of $Ho(\A)$ is composed of morphisms in $\A$ together with formal inverses $w^{-1}$ where $w \in \W$.  It follows that the family $\{\eta_{A}:FA \to GA: A \in \A \}$ is natural in $Ho(\A)$ just when it is natural where restricted to $\A$.  Similarly given $S:\A^{op} \times \A \to \B$ and a family of morphisms $\{\theta_{A}:X \to S(A,A):A \in \A \}$ we can consider the class $Ex(\theta) \subseteq Mor(\A)$ with respect to which $\theta$ is extranatural.  This has the same closure properties as before.  It follows that the family $\theta_{A}:X \to S(A,A)$ is extranatural in $Ho(\A)$ just when it is extranatural when restricted to maps in $\A$.\\
Using this reasoning we deduce that ~\eqref{eq:RD2} and ~\eqref{eq:RD3} are natural.  We likewise obtain the naturality of the $L$-component  of ~\eqref{eq:RD1} in each variable.  So it suffices to show that $(k_{C,A},1):[[QC,QA],[QC,B]] \to [Q[QC,A],[QC,B]]$ is natural in each variable.  This follows from Diagram~\eqref{eq:k4} of Lemma~\ref{Theorem:LemmaK}.\\
We verify the diagrams (C1-C5) below.  Each involves an instance of the corresponding diagram (C1-C5) for the skew closed structure on $\C$ itself, an application of Lemma~\ref{Theorem:LemmaK} and straightforward applications of naturality.
\begin{itemize}
\item[(C2)]
\begin{equation*}
\def\objectstyle{\scriptstyle}
\def\labelstyle{\scriptstyle}
\cd{[Q[QA,A],[QA,C]] \ar[r]^{[Q[p,1],1]} & [Q[A,A],[QA,C]] \ar[r]^{[Qj,1]} & [QI,[QA,C]] \ar[dr]_{1} \ar[r]^{[Qq,1]^{-1}} & [QRI,[QA,C]] \ar[d]^{[Qq,1]} \\
& & & [QI,[QA,C]] \ar[d]^{[e,1]} \\
[[QA,QA],[QA,C]]\ar[uu]^{[k,1]} \ar[rrr]_{[j,1]} &&& [I,[QA,C]] \ar[d]^{i}\\
[QA,C]\ar[u]^{L} \ar[rrr]_{1} &&& [QA,C]}
\end{equation*}
\item[(C3)]
\begin{equation*}
\def\objectstyle{\scriptstyle}
\def\labelstyle{\scriptstyle}
\cd{RI \ar[r]^{q^{-1}} & I \ar[dr]_{j} \ar[r]^{j} & [B,B] \ar[d]^{L} \ar[r]^{[p,1]} & [QB,B] \ar[d]^{L} \\
 & & [[QA,B],[QA,B]] \ar[dr]_{[p,1]} \ar[r]^{[[1,p],1]} & [[QA,QB],[QA,B]] \ar[d]^{[k,1]}\\
&&& [Q[QA,B],[QA,B]]} 
\end{equation*}
\vspace{1cm}
\item[(C4)]
\begin{equation*}
\def\objectstyle{\scriptstyle}
\def\labelstyle{\scriptstyle}
\cd{
&& [[QRI,QB],[QRI,C]] \ar[r]^{[k,1]} \ar[dr]_{[1,[Qq,1]]} & [Q[QRI,B],[QRI,C]]\ar[dr]^{[1,[Qq,1]]} \\\
&&& [[QRI,QB],[QI,C]] \ar[r]^{[k,1]} & [Q[QRI,B],[QI,C]] \ar[dd]^{[1,[e,1]]}\\
[QB,C] \ar[uurr]^{L} \ar[ddd]_{[Qi,1]}\ar@/_0.8pc/[ddr]_{[i,1]} \ar[dr]^{L} \ar[rr]^{L} && [[QI,QB],[QI,C]] \ar[ur]^{[[Qq,1],1]} \ar[d]^{[1,[e,1]]} \ar[r]^{[k,1]} & [Q[QI,B],[QI,C]]\ar[d]^{[1,[e,1]]} \ar[ur]_{[Q[Qq,1],1]}\\
& [[I,QB],[I,C]]\ar[d]^{[1,i]} \ar[r]^{[[e,1],1]} & [[QI,QB],[I,C]] \ar[d]^{[1,i]} \ar[r]^{[k,1]} & [Q[QI,B],[I,C]] \ar[d]^{[1,i]} \ar[r]^{[Q[Qq,1],1]} & [Q[QRI,B],[I,C]] \ar[d]^{[1,i]} \\
& [[I,QB],C] \ar[dl]_{[k,1]} \ar[r]^{[[e,1],1]} & [[QI,QB],C] \ar[r]^{[k,1]} & [Q[QI,B],C] \ar[r]^{[Q[Qq,1],1]} & [Q[QRI,B],C] \\
[Q[I,B],C] \ar@/_1pc/[urrr]_{[Q[e,1],1]}
}
\end{equation*}
\vspace{1cm}
\item[(C5)]
\begin{equation*}
\def\objectstyle{\scriptstyle}
\def\labelstyle{\scriptstyle}
\cd{RI \ar[r]^{q^{-1}} & I \ar@/_{1.8pc}/[dd]_{1} \ar[d]_{j} \ar[r]^{j} & [RI,RI] \ar[d]^{[q,1]} \ar[r]^{[p,1]} & [QRI,RI] \ar[d]^{[Qq,1]}\\
& [I,I] \ar[d]_{i} \ar[r]^{[1,q]} & [I,RI] \ar[d]_{i} \ar[dr]_{1} \ar[r]^{[p,1]} & [QI,RI] \ar[d]^{[e,1]} \\
& I \ar[r]_{q} & RI & [I,RI]\ar[l]^{i}}
\end{equation*}
\end{itemize}
\thispagestyle{empty}
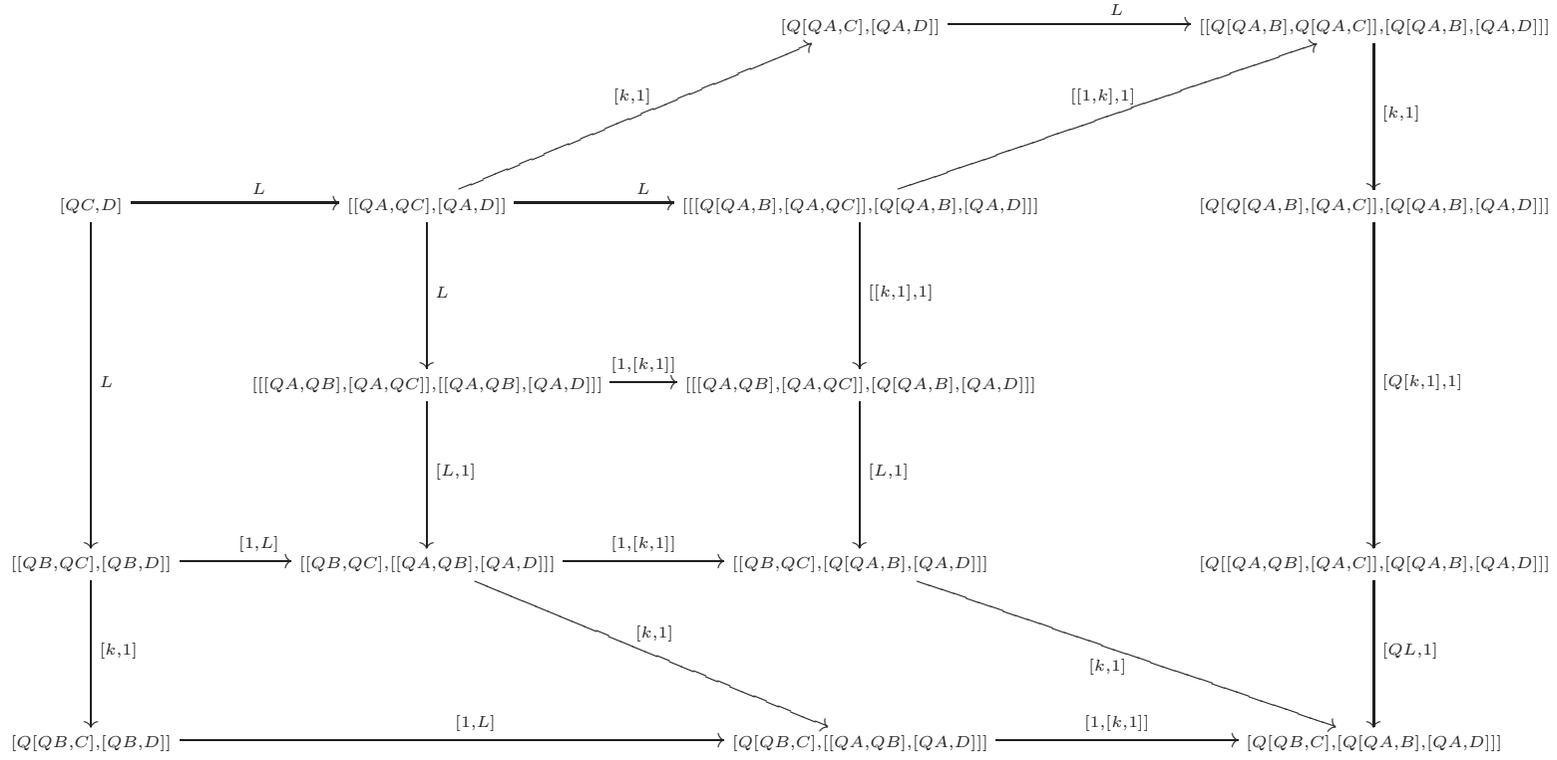
\begin{landscape}
\begin{itemize}
\item[(C1)]
\begin{equation*}
\def\objectstyle{\scriptscriptstyle}
\def\labelstyle{\scriptscriptstyle}
\cd{
& & [Q[QA,C],[QA,D]] \ar[rr]^{L} & & [[Q[QA,B],Q[QA,C]],[Q[QA,B],[QA,D]]] \ar[dd]^{[k,1]} \\\\
[QC,D] \ar[dddd]^{L} \ar[r]^{L}  & [[QA,QC],[QA,D]] \ar[dd]^{L} \ar[r]^{L} \ar[uur]^{[k,1]} & [[[Q[QA,B],[QA,QC]],[Q[QA,B],[QA,D]]] \ar[uurr]^{[[1,k],1]} \ar[dd]^{[[k,1],1]} & & [Q[Q[QA,B],[QA,C]],[Q[QA,B],[QA,D]]] \ar[dddd]^{[Q[k,1],1]} \\\\
& [[[QA,QB],[QA,QC]],[[QA,QB],[QA,D]]] \ar[r]^{[1,[k,1]]} \ar[dd]^{[L,1]} & [[[QA,QB],[QA,QC]],[Q[QA,B],[QA,D]]] \ar[dd]^{[L,1]} \\\\
[[QB,QC],[QB,D]] \ar[dd]^{[k,1]} \ar[r]^{[1,L]} & [[QB,QC],[[QA,QB],[QA,D]]] \ar[ddr]^{[k,1]} \ar[r]^{[1,[k,1]]} & [[QB,QC],[Q[QA,B],[QA,D]]] \ar[ddrr]_{[k,1]} &&  [Q[[QA,QB],[QA,C]],[Q[QA,B],[QA,D]]] \ar[dd]^{[QL,1]}\\\\
[Q[QB,C],[QB,D]] \ar[rr]^{[1,L]} && [Q[QB,C],[[QA,QB],[QA,D]]] \ar[rr]^{[1,[k,1]]} && [Q[QB,C],[Q[QA,B],[QA,D]]]}
\end{equation*}
\end{itemize}
\end{landscape}
We now transport the skew closed structure along the adjoint equivalence $Ho(j):Ho(\C_{f}) \leftrightarrows Ho(\C):Ho(R)$.  The skew closed structure obtained in this way has bifunctor
\begin{equation*}
\def\objectstyle{\scriptstyle}
\def\labelstyle{\scriptstyle}
\cd{Ho(\C)^{op} \times Ho(\C) \ar[rr]^<<<<<<<<<{Ho(R)^{op} \times Ho(R)}  & & Ho(\C_{f})^{op} \times Ho(\C_{f}) \ar[rr]^<<<<<<<{Ho([Q-,-])} && Ho(\C_{f}) \ar[r]^{Ho(j)} & Ho(\C)}
\end{equation*}
and unit given by
\begin{equation*}
\def\objectstyle{\scriptstyle}
\def\labelstyle{\scriptstyle}
\cd{I \ar[r]^<<<<{RI} & {Ho(\C_{f})}\ar[r]^{Ho(j)} & Ho(\C) \hspace{0.5cm} .}
\end{equation*}
So $Ho([QR-,R-])$ and $RI$ respectively.  Neither is quite as claimed.  We finally obtain the skew closed structure stated in the theorem by transferring this last skew closed structure along the isomorphisms of bifunctors $[Qq,1]:Ho([QR-,R-]) \to Ho([Q-,R-])$ and of units $q^{-1}:RI \to I$.
\end{proof}
We often refer to $(Ho(\C),[-,-]_{r},I)$ as the \emph{right derived} skew closed structure since $[-,-]_{r}$ is the right derived functor of $[-,-]$.
\begin{Definition}
Let $\C$ be a model category with a skew closed structure $(\C,[-,-],I)$ satisfying Axiom C.  We say that $(\C,[-,-],I)$ is \emph{homotopy closed} if the right derived skew closed structure $(Ho(\C),[-,-]_{r},I)$ is genuinely closed.
\end{Definition}
\begin{Proposition}\label{prop:hclosed2}
Let $(\C,[-,-],I)$ be a skew closed category satisfying Axiom C.  Then 
$(\C,[-,-],I)$ is homotopy closed if and only if the following two conditions 
are met.
\begin{enumerate}
\item For all cofibrant $A$ and fibrant $B$ the map $$v=\C(j,1) \circ [A,-] :\C(A,B) \to \C(I,[A,B])$$ is a bijection on homotopy classes of maps.  
\item For all fibrant $A$ the map $i:[I,A] \to A$ is a weak equivalence.
\end{enumerate}
\end{Proposition}
\begin{proof}
We will show that (1) and (2) amount to left and right normality of $(Ho(\C),[-,-]_{r},I)$ respectively.  Now $Ho(\C)$ is left normal just when
\begin{equation*}\label{eq:local}
\cd{
Ho(\C)(A,B) \ar[rrrr]^{Ho(\C)(j_{r},1) \circ [A,-]_{r}}  && & & Ho(\C)(I,[A,B]_{r})
}
\end{equation*}
is a bijection for all $A$ and $B$.  As in any skew closed category this map is natural in both variables.  Since we have isomorphisms $QA \to A$ and $B \to RB$ in $Ho(\C)$ the above map will be an isomorphism for all $A,B$ just when it is so for all cofibrant $A$ and fibrant $B$.  For such $A$ and $B$ we consider the diagram
\begin{equation*}
\cd{
\C(A,B) \ar[rrrr]^{\C(j,1) \circ \C([p,q],1) \circ [QA,R-]} \ar[d] &&  & & \C(I,[QA,RB]) \ar[d]\\
Ho(\C)(A,B) \ar[rrrr]^{[Ho(\C)(j_{r},1) \circ [A,-]_{r}} & & &&  Ho(\C)(I,[QA,RB])}
\end{equation*}
which is commutative by definition of $[A,-]_{r}$ and $j_{r}$.  The left and right vertical morphisms are surjective and identify precisely the homotopic maps.  It follows that the bottom row is invertible just when the top row induces a bijection on homotopy classes.  By naturality of $p$ and $q$ we can rewrite the top row as
\begin{equation*}
\cd{
\C(A,B) \ar[rrr]^{\C(j,1) \circ [A,-]} & & & \C(I,[A,B]) \ar[rr]^{\C(1,[p,q])} & & \C(I,[QA,RB]) \hspace{0.5cm} .}
\end{equation*}
Axiom C ensures that $[p,q]:[A,B] \to [QA,RB]$ is a weak equivalence between fibrant objects.  Since $I$ is cofibrant it follows that $\C(1,[p,q])$ is a bijection on homotopy classes.  Therefore the above composite is a bijection on homotopy classes just when its left component is.\\
Now $Ho(\C)$ is right normal just when $i_{r}=q^{-1} \circ i_{RA} \circ [e,1]:[QI,RA] \to [I,RA] \to RA \to A$ is invertible; equally just when $i_{RA}:[I,RA] \to RA$ is invertible for each $A$.  For fibrant $A$ we have $i_{A} \cong i_{RA}$ in $Ho(\C)^{\atwo}$ and the result follows.
\end{proof}

\subsubsection{The symmetric case}
Consider a skew closed structure $(\C,[-,-],I)$ satisfying Axiom C.  By Theorem~\ref{Theorem:hclosed1} we may form the right derived skew closed structure $(Ho(\C),[-,-]_{r},I)$.  Then from a symmetry isomorphism $s:[A,[B,C]] \cong [B,[A,C]]$ we can define a symmetry isomorphism $s_{r}:[A,[B,C]_{r}]_{r} \cong [B,[A,C]_{r}]_{r}$ on the right derived internal hom as below.

\begin{equation}\label{eq:DSym}
\def\objectstyle{\scriptstyle}
\def\labelstyle{\scriptstyle}
\cd{[QA,R[QB,RC]] \ar[r]^<<<<{[1,q]^{-1}} & [QA,[QB,RC]] \ar[r]^{s} & [QB,[QA,RC]] \ar[r]^{[1,q]} & [QB,R[QA,RC]]}
\end{equation}
\begin{Proposition}\label{prop:derivedSymmetry}
Consider $(\C,[-,-],I)$ satisfying Axiom C, so we have the skew closed structure $(Ho(\C),[-,-]_{r},I)$ of Theorem~\ref{Theorem:hclosed1}.\\
If a natural isomorphism $s:[A,[B,C]] \cong [B,[A,C]]$ satisfies any of $S1-S4$ then so does $s_{r}:[A,[B,C]_{r}]_{r} \cong [B,[A,C]_{r}]_{r}$.  In particular, if $(\C,[-,-],I,s)$ is symmetric skew closed then so is $(Ho(\C),[-,-]_{r},I,s_{r})$.
\end{Proposition}
\begin{proof}
As in the proof of Theorem~\ref{Theorem:hclosed1} we transport the structure from $Ho(\C_{f})$.  We can extend the skew closed structure $(Ho(\C_{f}), Ho[Q-,-],RI)$ described therein by a symmetry transformation $s_{r}$ whose component at $(A,B,C)$ is
\begin{equation*}
\cd{[QA,[QB,C]] \ar[rr]^{s_{QA,QB,C}} && [QB,[QA,C]]}
\end{equation*}
Since the components of $s_{r}$ are just those of $s$ it follows that $S1$ and $S2$ hold in $Ho(\C_{f})$ if they do so in $\C$.  The diagrams for $S3$ and $S4$ are below.
\begin{itemize}
\item[(S3)]
\begin{equation*}
\def\objectstyle{\scriptstyle}
\def\labelstyle{\scriptstyle}
\cd{
[QA,[QB,C]] \ar[dd]_{s} \ar[rr]^{L} && [[QD,QA],[QD,[QB,C]]] \ar[d]^{[1,s]} \ar[rr]^{[k,1]} && [Q[QD,A],[QD,[QB,C]]] \ar[d]^{[1,s]}\\
&& [[QD,QA],[QB,[QD,C]]] \ar[d]^{s} \ar[rr]^{[k,1]} && [Q[QD,A],[QB,[QD,C]]] \ar[d]^{s} \\
[QB,[QA,C]] \ar[rr]^{[1,L]} && [QB,[[QD,QA],[QD,C]]] \ar[rr]^{[1,[k,1]]} && [QB,[Q[QD,A],[QD,C]]]
}
\end{equation*}
\item[(S4)]
\begin{equation*}
\def\objectstyle{\scriptstyle}
\def\labelstyle{\scriptstyle}
\cd{
[QA,B] \ar[ddd]_{1} \ar[r]^<<<<{L} & [[QA,QA],[QA,B]] \ar[dd]^{[j,1]} \ar[r]^{[k,1]} & [Q[QA,A],[QA,B]] \ar[r]^{[Q[p,1],1]} &[Q[A,A],[QA,B]] \ar[d]^{[Qj,1]}\\
&&& [QI,[QA,B]]\ar[dl]_{1} \ar[d]^{[Qq,1]^{-1}}\\
& [I,[QA,B]] \ar[d]^{s} &[QI,[QA,B]] \ar[d]^{s} \ar[l]^{[e,1]} & [QRI,[QA,B]] \ar[l]^{[Qq,1]} \ar[d]^{s} \\
[QA,B] & [QA,[I,B]] \ar[l]^{[1,i]} & [QA,[QI,B]] \ar[l]^{[1,[e,1]]} & [QA,[QRI,B]] \ar[l]^{[1,[Qq,1]]}}
\end{equation*}
\end{itemize}
Therefore  $(Ho(\C_{f}), Ho[Q-,-],RI,s_{r})$ satisfies any of $S1-S4$ when $(\C,[-,-],I,s)$ does so.  The desired structure on $Ho(\C)$  is obtained by transporting from the equivalent $Ho(\C_{f})$ as in Theorem~\ref{Theorem:hclosed1}.
\end{proof}

\subsection{Monoidal skew closed structure on the homotopy category and the homotopical version of Eilenberg and Kelly's theorem}
Let $\C$ be a model category equipped with a monoidal skew closed structure $(\C,\otimes,[-,-],I)$.  In order to derive the tensor-hom adjunctions $- \otimes QA \dashv [QA,-]$ to the homotopy category, we will make use of the concept of a Quillen adjunction.\\
An adjunction $F:\C \leftrightarrows D:U$ of model categories is said to be a \emph{Quillen adjunction} if the right adjoint $U:\D \to \C$ preserves fibrations and trivial fibrations.  One says that $U$ is \emph{right Quillen}.  This is equivalent to asking that $F$ preserves cofibrations and trivial cofibrations, in which case $F$ is said to be \emph{left Quillen}.  The derived functors $F_{l}$ and $U_{r}$ then exist and form an adjoint pair $F_{l}:Ho(\C) \leftrightarrows Ho(\D):U_{r}$.  This works as follows.  At cofibrant $A$ and fibrant $B$ the isomorphisms $\phi_{A,B}:\D(FA,B) \cong \C(A,UB)$ are well defined on homotopy classes and so give natural bijections $$\phi_{A,B}:Ho(\D)(FA,B) \cong Ho(\C)(A,UB)$$ -- we use the same labelling.  At $A,B \in Ho(\C)$ the hom-set bijections $$\phi^{d}_{A,B}:Ho(\D)(F_{l}A,B) \cong Ho(\C)(A,U_{r}B)$$ are then given by conjugating the $\phi_{A,B}$ as below:
\begin{equation*}
\def\objectstyle{\scriptstyle}
\def\labelstyle{\scriptstyle}
\cd{Ho(\D)(FQA,B) \ar[r]^{(1,q)} & Ho(\D)(FQA,RB) \ar[r]^{\phi_{A,B}} & Ho(\C)(QA,URB) \ar[r]^{(p,1)^{-1}} & Ho(\C)(A,URB)}
\end{equation*}
It follows that the unit of the derived adjunction -- the \emph{derived unit} -- is given by 
\begin{equation*}
\cd{A \ar[rr]^{p_{A}^{-1}} && QA \ar[rr]^{\eta_{QA}} && UFQA \ar[rr]^{Uq_{FQA}} && URFQA}
\end{equation*}
whilst the derived counit admits a dual description.
\begin{Proposition}\label{prop:AxiomGood}
For $\C$ a model category and $(\C,\otimes,[-,-],I)$ monoidal skew closed the following are equivalent.
\begin{enumerate}
 \item For cofibrant $X$ the functor $- \otimes X$ is left Quillen and for each cofibrant $Y$ the functor $Y \otimes -$ preserves weak equivalences between cofibrant objects.
\item For cofibrant $X$ the functor $[X,-]$ is right Quillen and for fibrant $Y$ the functor $[-,Y]$ preserves weak equivalences between cofibrant objects.
\end{enumerate}
\end{Proposition}
\begin{proof}
Certainly $- \otimes X$ is left Quillen just when $[X,-]$ is right Quillen. Let $f:A \to B$ be a weak equivalence between cofibrant objects.  Then the natural transformation of left Quillen functors $- \otimes f:- \otimes A \to - \otimes B$ and the natural transformation of right Quillen functors $[f,-]:[B,-] \to [A,-]$ are \emph{mates}.  By Corollary 1.4.4(b) of \cite{Hovey1999Model}  it follows that $Y \otimes f$ is a weak equivalence for all cofibrant $Y$ if and only if $[f,Y]$ is a weak equivalence for all fibrant $Y$.
\end{proof}
\begin{AxiomMC}
$(\C,\otimes,[-,-],I)$ satisfies either of the equivalent conditions of Proposition~\ref{prop:AxiomGood} and the unit $I$ is cofibrant.
\end{AxiomMC}
\begin{Proposition}\label{prop:axiomImplication}
Axiom MC implies Axioms M and C.
\end{Proposition}
\begin{proof}
When Axiom MC holds Ken Brown's lemma (1.1.12 of \cite{Hovey1999Model}) ensures that for cofibrant $X$ the functor $X \otimes -$ preserves weak equivalences between cofibrant objects.  So Axiom MC is a strengthening of Axiom M.  That it implies Axiom C is clear: if the right adjoint $[X,-]$ preserves fibrations it preserves fibrant objects.
\end{proof}
For $(\C,\otimes,[-,-],I,\phi)$ satisfying Axiom MC it follows that we can form the left derived skew monoidal and right derived skew closed structures on $Ho(\C)$.  The following result establishes, as expected, that these form part of a monoidal skew closed structure on $Ho(\C)$. 
\begin{Theorem}\label{Theorem:TotalDerived}
Let $(\C,\otimes,[-,-],I,\phi)$ be a monoidal skew closed category satisfying Axiom MC.
Then the left derived skew monoidal structure $(Ho(\C),\otimes_{l},I)$ and the right derived skew closed structure $(\C,[-,-]_{r},I)$ together with the isomorphisms $$\phi^{d}:Ho(\C)(QA \otimes QB,C) \cong Ho(\C)(A,[QB,RC])$$
form a monoidal skew closed structure on $Ho(\C)$.
\end{Theorem}
The straightforward but long proof is deferred until the appendix.  The following result is the homotopical version of Eilenberg and Kelly's theorem.  Note that by Proposition~\ref{prop:hclosed2} conditions (1) and (2) amount to $(\C,[-,-],I)$ being homotopy closed.

\begin{Theorem}\label{Theorem:hEK}
Let $(\C,\otimes,[-,-],I)$ be a monoidal skew closed category satisfying Axiom MC.  Then $(\C,\otimes,I)$ is homotopy monoidal if and only if the following three conditions are satisfied.
\begin{enumerate}
\item For all cofibrant $A$ and fibrant $B$ the function $v: \C(A,B) \to \C(I,[A,B])$ is a bijection on homotopy classes of maps,
\item For all fibrant $A$ the map $i:[I,A] \to A$ is a weak equivalence,
\item The transformation $t:[A \otimes B,C] \to [A,[B,C]]$ 
is a weak equivalence whenever $A$ and $B$ are cofibrant and $C$ is fibrant.  
\end{enumerate}
\end{Theorem}
\begin{proof}
Combining Theorems'~\ref{Theorem:TotalDerived} and~\ref{Theorem:EK} we have that $(Ho(\C),\otimes_{l},I)$ is monoidal just when $(Ho(\C),[-,-]_{r},I)$ is closed and the induced transformation $t^{d}:[QA \otimes QB,RC] \to [QA,R[QB,RC]]$ is an isomorphism in $Ho(\C)$.  By Proposition~\ref{prop:hclosed2} closedness amounts to (1) and (2) above.  From the proof of Theorem~\ref{Theorem:TotalDerived} the transformation $t^{d}$ is given by the composite
\begin{equation*}
\def\objectstyle{\scriptstyle}
\def\labelstyle{\scriptstyle}
\cd{
[Q(QA \otimes QB),RC] \ar[r]^<<<<<{[p,1]^{-1}} & [QA \otimes QB,RC] \ar[r]^{t} & [QA,[QB,RC]] \ar[r]^{[1,q]} & [QA,R[QB,RC]]}
\end{equation*}
and therefore is invertible just when the central component 
\begin{equation*}
t_{QA,QB,RC}:[QA \otimes QB,RC] \to [QA,[QB,RC]]
\end{equation*}
is so for all $A,B$ and $C$.  For $A,B$ cofibrant and $C$ fibrant $t_{QA,QB,RC}$ is isomorphic to $t_{A,B,C}:[A \otimes B,C] \to [A,[B,C]]$.  Since $QA$ and $QB$ are cofibrant and $RC$ is fibrant it follows that $t_{QA,QB,RC}$ is invertible for all $A,B,C$ just when $t_{A,B,C}$ is invertible for all cofibrant $A$,$B$ and fibrant $C$.
\end{proof}
Again we have a symmetric variant.
\begin{Theorem}\label{Theorem:hEKs}
Let $(\C,\otimes,[-,-],I)$ be a monoidal skew closed category satisfying Axiom MC.
\begin{enumerate}
\item If $(\C,[-,-],I)$ is homotopy closed and admits a natural symmetry isomorphism $s:[A,[B,C]] \cong [B,[A,C]]$ satisfying $S3$ then $(\C,\otimes,I)$ is homotopy monoidal.  
\item If, in addition to (1), $(\C, [-,-],I,s)$ is symmetric skew closed then $(\C,\otimes,I)$ is homotopy symmetric monoidal.
\end{enumerate}
\end{Theorem}
\begin{proof}
By Proposition~\ref{prop:derivedSymmetry} if $s$ satisfies S3 then so does $s_{r}$ with respect to $(Ho(\C),[-,-]_{r},I)$.  From Theorem~\ref{Theorem:ClassicalSymmetry} it follows that $(Ho(\C),\otimes_{l},I)$ is monoidal.  The second part follows again by application of Proposition~\ref{prop:derivedSymmetry} and ~\ref{Theorem:ClassicalSymmetry}.
\end{proof}

\section{Pseudo-commutative 2-monads and monoidal bicategories}\label{section:pseudoCommutative}
In the category $CMon$ of commutative monoids the set $CMon(A,B)$ forms a commutative monoid $[A,B]$ with respect to the pointwise structure of $B$.  This is the internal hom of a symmetric monoidal closed structure on $CMon$ whose tensor product represents functions $A \times B \to C$ that are homomorphisms in each variable.  From the monad-theoretic viewpoint the enabling property is that the commutative monoid monad on $\Set$ is a \emph{commutative monad}.\\
Extending this to dimension 2, Hyland and Power \cite{Hyland2002Pseudo} introduced the notion of a \emph{pseudo-commutative 2-monad} $T$ on $\Cat$.  Examples include the 2-monads for categories with a class of limits, permutative categories, symmetric monoidal categories and so on.  For such $T$ they showed that the 2-category of strict algebras and pseudomorphisms admits the structure of a \emph{pseudo-closed 2-category} -- a slight weakening of the notion of a closed category with a 2-categorical element.  Theorem 2 of \emph{ibid.} described a bicategorical version of Eilenberg and Kelly's theorem, designed to produce a monoidal bicategory structure on $\TAlg$.  However they did not give the details of the proof, which involved lengthy calculations of a bicategorical nature, and expressed their dissatisfaction with the argument.\begin{footnote}{ From \cite{Hyland2002Pseudo}:``Naturally, we are unhappy with the proof we have just outlined. Since the data we start from is in no way symmetric we expect some messy difficulties: but the calculations we do not give are very tiresome, and it would be only too easy to
have made a slip.  Hence we would like a more conceptual proof."}\end{footnote}\\
In this section we take a slightly different route to the monoidal bicategory structure on $\TAlg$.  We begin by making minor modifications to Hyland and Power's construction to produce a skew closed structure on the 2-category $\TAlgs$ of algebras and \emph{strict morphisms.}  This is simply the restriction of the pseudo-closed 2-category structure on $\TAlg$.  We then obtain a monoidal skew closed structure on $\TAlgs$  and, using Theorem~\ref{Theorem:hEK}, establish that it is homotopy monoidal. The monoidal bicategory structure on $\TAlg$ is obtained by transport of structure from the full sub 2-category of $\TAlgs$ containing the cofibrant objects.  

\subsection{Background on commutative monads}
If $ V$ is a symmetric monoidal closed category and $T$ an endofunctor of $V$ then enrichments of $T$ to a $V$-functor correspond to giving a \emph{strength}: that is, a natural transformation $t_{A,B}:A \otimes TB \to T(A \otimes B)$ subject to associativity and identity conditions.  One obtains a \emph{costrength} $t^{*}_{A,B}:TA \otimes B \to T(A \otimes B)$ related to the strength by means of the symmetry isomorphism $c_{A,B}:A \otimes B \to B \otimes A$.\\
If $(T,\eta,\mu)$ is a $V$-enriched monad then we can consider the following diagram
\begin{equation}\label{eq:commutative}
\xy
(0,0)*+{TA \otimes TB}="00";(30,0)*+{T(TA \otimes B)}="10";(60,0)*+{T^{2}(A \otimes B)}="20";
(0,-20)*+{T(A \otimes TB)}="01";(30,-20)*+{T^{2}(A \otimes B)}="11";(60,-20)*+{T(A \otimes B)}="21";
{\ar^{t} "00";"10"};
{\ar^{Tt^{\star}} "10";"20"};
{\ar^{\mu} "20";"21"};
{\ar_{t^{\star}} "00";"01"};
{\ar^{Tt} "01";"11"};
{\ar^{\mu} "11";"21"};
\endxy
\end{equation}
and if this commutes for all $A$ and $B$ then $T$ is said to be a \emph{commutative monad} \cite{Kock1970Monads}.\\
Now if $T$ is commutative and $V$ sufficiently complete and cocomplete then the category of algebras $V^{T}$ is itself symmetric monoidal closed \cite{Kock1971Closed, Keigher1978Symmetric}.  Both tensor product and internal hom represent \emph{$T$-bilinear maps} -- this perspective was explored in \cite{Kock1971Bilinearity} and more recently in \cite{Seal2013Tensors}.  More generally, a $T$-multilinear map consists of a morphism $f:A_{1}\otimes \ldots \otimes A_{n} \to B$ which is a \emph{$T$-algebra map in each variable.}  This means that the diagram
$$
\xy
(0,0)*+{A_{1} \otimes \ldots \otimes TA_{i}\otimes  \ldots \otimes  A_{n}}="00";(50,0)*+{T(A_{1} \otimes \ldots \otimes A_{n})}="10";(90,0)*+{TB}="20";
(0,-20)*+{A_{1} \otimes \ldots \otimes A_{i}\otimes  \ldots \otimes  A_{n}}="01";(90,-20)*+{B}="21";
{\ar^<<<<<<{t} "00";"10"};
{\ar^<<<<<<{Tf} "10";"20"};
{\ar^{b} "20";"21"};
{\ar_{1 \otimes \ldots \otimes a_{i} \otimes \ldots \otimes 1} "00";"01"};
{\ar^{f} "01";"21"};
\endxy$$
is commutative for each $i$ where the top row $t:A_{1} \otimes \ldots \otimes TA_{i}\otimes  \ldots \otimes  A_{n} \to T(A_{1} \otimes \ldots \otimes A_{n})$ is the unique map constructible from the strengths and costrengths.  $T$-multilinear maps form the morphisms of a \emph{multicategory} of $T$-algebras.  Surprisingly, the multicategory perspective appears to have first been explored in the more general 2-categorical setting of \cite{Hyland2002Pseudo}.

\subsection{Background on 2-monads}\label{section:2-monads}
The category of small categories $\Cat$ is cartesian closed and hence provides a basis suitable for enriched category theory.   In particular one has the notions of $\Cat$-enriched category -- hence 2-category -- and of $\Cat$-enriched monads -- hence $2$-monads.  The appendage ``2-" will always refer to strict $\Cat$-enriched concepts.\\
Given a 2-monad $T=(T,\eta, \mu)$ on a 2-category $\C$ one has the Eilenberg-Moore 2-category $\TAlgs$ of algebras.  In $\TAlgs$ everything is completely strict. There are the usual \emph{strict} algebras $\f A=(A,a)$ satisfying $a \circ Ta = a \circ \mu_{A}$ and $a \circ \eta_{A}=1$.   The \emph{strict} morphisms $f:\f A \to \f B$ of $\TAlgs$ satisfy the usual equation $b \circ Tf = f \circ a$ on the nose, whilst the 2-cells $\alpha:f \Rightarrow g \in \TAlgs(\f A,\f B)$ satisfy $b \circ T\alpha = \alpha \circ a$.\\
$\TAlgs$ is just as well behaved as its $\Set$-enriched counterpart. Important facts for us are the following ones.
\begin{itemize}
\item The usual (free, forgetful)-adjunction lifts to a 2-adjunction $F \dashv U$ where $U:\TAlgs \to \C$ is the evident forgetful 2-functor. 
\item Suppose that $\C$ is a locally presentable 2-category: one, like $\Cat$, that is cocomplete in the sense of enriched category theory \cite{Kelly1982Basic} and whose underlying category is locally presentable \cite{Adamek1994Locally}. If $T$ is accessible --  preserves $\lambda$-filtered colimits for some regular cardinal $\lambda$ -- then $\TAlgs$ is also locally presentable.  
\end{itemize}
There are accessible 2-monads $T$ on $\Cat$ whose strict algebras are categories with $\f D$-limits, permutative categories, symmetric monoidal categories and so on.  In particular the examples of skew closed 2-categories from Section ~\ref{section:structuredCats} reside on 2-categories of the form $\TAlgs$ for $T$ an accessible 2-monad on $\Cat$.\\
So far we have discussed strict aspects of two-dimensional monad theory. Though there are several possibilities, the only weak structures of interest here are \emph{pseudomorphisms} of \emph{strict} $T$-algebras.  A pseudomorphism $\f f:\f A \rightsquigarrow \f B$ consists of a morphism $f:A \to B$ and invertible 2-cell $\overline{f}:b \circ Tf \cong f \circ a$ satisfying two coherence conditions \cite{Blackwell1989Two-dimensional}.  These are the morphisms of the 2-category $\TAlg$ into which $\TAlgs$ includes via an identity on objects 2-functor $\iota:\TAlgs \to \TAlg$.  The inclusion commutes with the forgetful 2-functors
\begin{equation}
\cd{\TAlgs \ar[dr]_{U} \ar[rr]^{\iota} & & \TAlg \ar[dl]^{U} \\
& \C}
\end{equation}
to the base.  Pseudomorphisms of $T$-algebras capture functors preserving categorical structure up to isomorphism.  For example, in the case that $T$ is the 2-monad for categories with $\f D$-limits or permutative categories we obtain the 2-categories $\DLim$ and $\Perm$ as $\TAlg$.\\
An important tool in the study of pseudomorphisms are \emph{pseudomorphism classifiers}.  If $T$ is a reasonable 2-monad -- for instance, an accessible 2-monad on $\Cat$ -- then by Theorem 3.3 of \cite{Blackwell1989Two-dimensional} the inclusion $\iota:\TAlgs \to \TAlg$ has a left 2-adjoint $Q$.  We call $Q\f A$ the pseudomorphism classifier of $\f A$ since each pseudomorphism $\f f:\f A \rightsquigarrow \f B$ factors uniquely through the unit $\f q_{\f A}:\f A \rightsquigarrow Q\f A$ as a strict morphism $Q\f A \to \f B$.  The counit $p_{\f A}:Q\f A \to \f A$ is a strict map with homotopy theoretic content -- see Section~\ref{section:h2monads} below.

\subsection{From pseudo-commutative 2-monads to monoidal skew closed 2-categories}
Given a 2-monad $T$ on $\Cat$ we have, in particular, the corresponding strengths $t:T(A \times B) \to A \times TB$ and costrengths $T(A \times B) \to TA \times B$ and can enquire as to whether $T$ is commutative.  For those structures -- such as categories with finite products or symmetric monoidal categories  -- that involve an aspect of weakness in their definitions the relevant diagram ~\eqref{eq:commutative} rarely commutes on the nose, but often commutes up to natural isomorphism.  This leads to the notion of a \emph{pseudo-commutative 2-monad} $T$ which is a 2-monad $T$ equipped with invertible 2-cells
$$
\xy
(0,0)*+{TA \times TB}="00";(30,0)*+{T(TA \times B)}="10";(60,0)*+{T^{2}(A \times B)}="20";
(0,-20)*+{T(A \times TB)}="01";(30,-20)*+{T^{2}(A \times B)}="11";(60,-20)*+{T(A \times B)}="21";
{\ar^{t} "00";"10"};
{\ar^{Tt^{\star}} "10";"20"};
{\ar^{\mu} "20";"21"};
{\ar_{t^{\star}} "00";"01"};
{\ar_{Tt} "01";"11"};
{\ar_{\mu} "11";"21"};
{\ar@{=>}^{\alpha_{A,B}}(30,-5)*+{};(30,-15)*+{}};
\endxy$$
subject to axioms (see Definition 5 of \cite{Hyland2002Pseudo}) asserting the equality of composite 2-cells built from the above ones.  If $\alpha$ commutes with the symmetry isomorphism -- in the sense that $\alpha_{B,A}=Tc_{A,B} \circ \alpha_{A,B} \circ c_{TB,TA}$ -- then $T$ is said to be a \emph{symmetric} pseudo-commutative 2-monad.\\
The 2-monad for categories with $\f D$-limits is symmetric pseudo-commutative \cite{Lopez-Franco2011Pseudo} as are the 2-monads for permutative and symmetric monoidal categories \cite{Hyland2002Pseudo}.  An example of a pseudo-commutative 2-monad which is not symmetric is the 2-monad for braided strict monoidal categories \cite{Corner2013Operads}.
\subsubsection{The 2-multicategory of algebras}
For $T$ pseudo-commutative one can define $T$-multilinear maps.  A $T$-multilinear map $\f f:(\f A_{1}, \ldots ,\f A_{n}) \to \f B$ consists of a functor $f:A_{1} \times \ldots A_{n} \to B$ together with a family of invertible 2-cells $f_{i}$:
\begin{equation*}
\xy
(0,0)*+{A_{1} \times \ldots \times TA_{i}\times  \ldots \times  A_{n}}="00";(50,0)*+{T(A_{1} \times \ldots \times A_{n})}="10";(90,0)*+{TB}="20";
(0,-20)*+{A_{1} \times \ldots \times A_{i}\times  \ldots \times  A_{n}}="01";(90,-20)*+{B}="21";
{\ar^{t} "00";"10"};
{\ar^<<<<<<{Tf} "10";"20"};
{\ar^{c} "20";"21"};
{\ar_{1 \times \ldots \times a_{i} \times \ldots \times 1} "00";"01"};
{\ar^{f} "01";"21"};
{\ar@{=>}^{f_{i}}(45,-5)*+{};(45,-15)*+{}};
\endxy
\end{equation*}
satisfying indexed versions of the pseudomorphism equations, and a compatibility condition involving the pseudo-commutativity.  A nullary map $(-) \to \f B$ is defined to be an object of the category $B$.\\
There are transformations of multilinear maps and these are the morphisms of a category $\TALG(\f A_{1},\f A_{2} \ldots \f A_{n};\f B)$.  Proposition 18 of \emph{ibid.} shows that these are the hom-categories of a 2-multicategory of $T$-algebras $\TALG$ and that, moreover, if $T$ is symmetric pseudo-commutative then $\TALG$ is a symmetric 2-multicategory. $\TAlg$ is itself recovered as the 2-category of unary maps.\\
Of course we can speak of multimaps $(\f A_{1}, \ldots ,\f A_{n}) \to \f B$ which are strict in $\f A_{i}$: those for which the natural transformation $f_{i}$ depicted above is an identity.  Note that this agrees with the formulation given in Definition~\ref{Thm:strict}.
\begin{Theorem}[Hyland-Power \cite{Hyland2002Pseudo}]\label{Thm:HpMulti}
The 2-multicategory $\TALG$ is closed.  Moreover a multimap $(\f A_{1},\f A_{2} \ldots \f A_{n},\f B) \to \f C$ is strict in $\f A_{i}$ just when the corresponding map $(\f A_{1},\f A_{2} \ldots \f A_{n}) \to [\f B,\f C]$ is so.
\end{Theorem}

\subsubsection{The skew closed structure}\label{section:PseudoSkew}
By definition $\TALG(-;\f A)= A$.  Since we have a natural isomorphism $\TAlgs(F1,\f A) \cong \Cat(1,A) \cong A$ and a suitable closed 2-multicategory $\TALG$ we can apply Theorem~\ref{Thm:multiToClosed2} to obtain a skew closed structure on $\TAlgs$.

\begin{Theorem}\label{Theorem:SkewAlgebras}
Let $T$ be a pseudo-commutative 2-monad on $\Cat$.  
\begin{enumerate}
\item
Then $(\TAlg,[-,-],L)$ is a semi-closed 2-category.  Moreover $[-,-]$ and $L$ restrict to $\TAlgs$ where they form part of a skew closed 2-category \newline $(\TAlgs,[-,-],F1,L,i,j)$.
\item If $T$ is symmetric then $(\TAlg,[-,-],L)$ is symmetric semi-closed and \newline $(\TAlgs,[-,-],F1,L,i,j)$ is symmetric skew closed.
\end{enumerate}
\end{Theorem}
The skew closed 2-category $(\TAlgs,[-,-],F1,L,i,j)$ has components as constructed in Section~\ref{section:multi}.  Let us record, for later use, some further information about these components.
\begin{enumerate}
\item The underlying category of $[\f A, \f B]$ is just $\TAlg(\f A, \f B)$.  More generally
$U \circ [-,-]=\TAlg(-,-):\TAlg^{op} \times \TAlg \to \Cat$.
\item The underlying functor of $L:[\f A,\f B] \to [[\f C,\f A],[\f C,\f B]]$ is given by $[\f C,-]_{\f A,\f B}:\TAlg(\f A,\f B) \to \TAlg([\f C,\f A],[\f C,\f B]).$
\item The underlying functor of $i:[F1,\f A] \to \f A$ is the composite
\begin{equation*}
\cd{\TAlg(\f {F1},\f A) \ar[rr]^{U_{\f {F1},\f A}} && \Cat(T1,A) \ar[rr]^{\Cat(\eta_{1},A)} && \Cat(1,A) \ar[rr]^{ev_{\bullet}} &&  A}
\end{equation*}
whose last component is the evaluation isomorphism.  
\item $j:F1 \to [\f A,\f A]$ is the transpose of the functor $\hat{1}:1 \to \TAlg(\f A,\f A)$ selecting the identity on $\f A$.
\end{enumerate}
(1) follows from the construction of the hom algebra $[\f A, \f B]$ in \cite{Hyland2002Pseudo} as a 2-categorical limit in $\TAlg$ created by $U:\TAlg \to \Cat$.\begin{footnote}{ This construction is accomplished in three stages by firstly forming an iso-inserter and then a pair of equifiers and amounts to the construction of a descent object.}\end{footnote}  Theorem 11 of \emph{ibid.} gives a full description of the isomorphisms $\TALG(\f A_{1}, \ldots, \f A_{n},\f B;\f C) \cong \TALG(\f A_{1} \ldots \f A_{n};[\f B,\f C])$.  From this, it follows that the evaluation multimap $ev:([\f A,\f B],\f A) \to \f B$ has underlying functor $\TAlg(\f A,\f B) \times \f A \to \f B$ acting by application, which is what is required for (3).  (2) follows from the analysis, given in Proposition 21 of \emph{ibid}, of how the same adjointness isomorphisms behave with respect to underlying maps.  (4) is by definition.

\subsubsection{The monoidal skew closed structure on $\TAlgs$}
We now describe left 2-adjoints to the 2-functors $[\f A,-]:\TAlgs \to \TAlgs$.  For this let us further suppose that $T$ is an accessible 2-monad.  We must show that each $\f B$ admits a reflection $\f B \otimes \f A$ along $[\f A,-]$.  Since $\TAlgs$ is cocomplete the class of algebras admitting such a reflection is closed under colimits; because each algebra is a coequaliser of frees it therefore suffices to show that each free algebra admits a reflection.  With this is mind observe that the triangle
\begin{equation}\label{eq:triangle}
\cd{\TAlgs \ar[dr]_{\TAlg(\f A,\iota-)} \ar[rr]^{[\f A,-]} && \TAlgs \ar[dl]^{U} \\
& \Cat}
\end{equation}
commutes.  Because $T$ is accessible we have the 2-adjunction $Q \dashv \iota$ and corresponding isomorphism $\TAlgs(Q\f A,-) \cong \TAlg(\f A,\iota-)$.  Now the representable $\TAlgs(Q\f A,-)$ has a left adjoint $-.Q\f A$ given by taking copowers.  It follows that at $C \in \Cat$ the reflection $FC \otimes \f A$ is given by $C.Q\f A$.   We conclude:
\begin{Proposition}
If $T$ is an accessible pseudo-commutative 2-monad on $\Cat$ then each $[\f A,-]:\TAlgs \to \TAlgs$ has a left 2-adjoint $- \otimes \f A$.  In particular $(\TAlgs, \otimes, [-,-], F1)$ is a monoidal skew closed 2-category.
\end{Proposition}
\subsection{From $\TAlgs$ as a monoidal skew closed 2-category to $\TAlg$ as a monoidal bicategory}
Our next goal is to show that $(\TAlgs, \otimes, [-,-], F1)$ is homotopy monoidal.  In order to do so requires understanding the Quillen model structure on $\TAlgs$ and its relationship with pseudomorphisms.  We summarise the key points below and refer to the original source \cite{Lack2007Homotopy-theoretic} for further details.
\subsubsection{Homotopy theoretic aspects of 2-monads}\label{section:h2monads}
Thought of as a mere category, $\Cat$ admits a Quillen model structure in which the weak equivalences are the equivalences of categories.  The cofibrations are the injective on objects functors and the fibrations are the isofibrations: functors with the isomorphism lifting property.  It follows that all objects are cofibrant and fibrant.\\
Equipped with the cartesian closed structure, $\Cat$ is a monoidal model category \cite{Hovey1999Model}.  Therefore one can speak of model 2-categories, of which $\Cat$ is the leading example.  It was shown in Theorem 4.5 of \cite{Lack2007Homotopy-theoretic} that for an accessible 2-monad $T$ on $\Cat$ the model structure lifts along $U:\TAlgs \to \Cat$ to a model 2-category structure on $\TAlgs$:  a morphism of $\TAlgs$ is a weak equivalence or fibration just when its image under $U$ is one.  It follows immediately that the adjunction $F \dashv U:\TAlgs \leftrightarrows \Cat$ is a Quillen adjunction.\\
Since $F$ preserves cofibrations each free algebra is cofibrant.  In fact, the cofibrant objects are the \emph{flexible algebras} of \cite{Blackwell1989Two-dimensional} and were studied long before the connection with model categories was made in \cite{Lack2007Homotopy-theoretic}.  Another source of cofibrant algebras comes from pseudomorphism classifiers: each $Q\f A$ is cofibrant.  In fact the counit $p_{\f A}:Q\f A \to \f A$ of the adjunction $Q \dashv \iota:\TAlgs \leftrightarrows \TAlg$ is a trivial fibration in $\TAlgs$; thus $Q\f A$ is a cofibrant replacement of $\f A$.\\
Theorem 4.7 of  \cite{Blackwell1989Two-dimensional} ensures that if $\f A$ is flexible then, for all $\f B$, the fully faithful inclusion $$\iota_{\f A,\f B}:\TAlgs(\f A,\f B) \to \TAlg(\f A,\f B)$$ is essentially surjective on objects: that is, an \emph{equivalence of categories}.  This important fact can also be deduced from the model 2-category structure: the inclusion $\iota_{\f A,\f B}$ is isomorphic to $\TAlgs(p_{\f A},\f B):\TAlgs(\f A,\f B) \to \TAlgs(Q\f A, \f B)$ which is an equivalence since $p_{\f A}:Q\f A \to \f A$ is a weak equivalence of cofibrant objects.\\
Finally we note that a parallel pair of algebra morphisms $f,g:\f A \rightrightarrows \f B$ are right homotopic just when they are isomorphic in $\TAlgs(\f A,\f B)$.  This follows from the fact that for each algebra $\f B$ the power algebra $[I,\f B]$ is a path object where $I$ is the walking isomorphism.  In particular, if $A$ is cofibrant then $f$ and $g$ are homotopic just when they are isomorphic.
\subsubsection{Homotopical behaviour of the skew structure}
\begin{Theorem}
Let $T$ be an accessible pseudo-commutative 2-monad on $\Cat$.
\begin{enumerate}
\item
Then the monoidal skew closed 2-category $(\TAlgs, \otimes, [-,-], F1)$ satisfies Axiom $MC$ and $(\TAlgs, \otimes, F1)$ is homotopy monoidal.  
\item If $T$ is symmetric then $(\TAlgs, \otimes, F1)$ is homotopy symmetric monoidal.
\end{enumerate}
\end{Theorem}
\begin{proof}
We observed above that each free algebra is cofibrant.  Therefore the unit $F1$ is cofibrant.  We now verify Axiom MC in its closed form: in fact we establish the stronger result that for all $\f A$ the 2-functor $[\f A,-]$ is right Quillen and that $[-,\f A]$ preserves all weak equivalences.  For the first part consider the equality $U \circ [\f A,-]=\TAlgs(\f A,\iota -)$ of \eqref{eq:triangle}.  Since $U$ reflects weak equivalences and fibrations it suffices to show that $\TAlg(\f A,\iota -):\TAlgs \to \Cat$ is right Quillen.  Now  $\TAlg(\f A,\iota -) \cong \TAlgs(Q\f A,-):\TAlgs \to \Cat$ and this last 2-functor is right Quillen since $Q\f A$ is cofibrant and $\TAlgs$ a model 2-category.\\
For the second part we use the commutativity $U \circ [-,\f A] \cong \TAlg(\iota -,\f A)$.  Arguing as before it suffices to show that $\TAlg(\iota -,\f A):\TAlgs \to \Cat$ preserves all weak equivalences or, equally, that the isomorphic $\TAlgs(Q\iota -,\f A)$ does so.  Now if $f:\f B \to \f C$ is a weak equivalence then $Q\iota f$ is a weak equivalence of cofibrant objects. As $\f A$, like all objects, is fibrant and $\TAlgs$ a model 2-category the functor $\TAlgs(Q\iota f,\f A)$ is an equivalence.\\
We now apply Theorem~\ref{Theorem:hEK} to establish that $\TAlgs$ is homotopy monoidal.  To verify the three conditions requires only the information on the underlying functors of $[-,-]$, $L$ and $i$ given in Section~\ref{section:PseudoSkew}.  Firstly we must show that the underlying function of  $$v_{\f A,\f B}:\TAlgs(\f A,\f B) \to \TAlgs(F1, [\f A,\f B])$$ induces a bijection on homotopy classes of maps for cofibrant $\f A$.  Since morphisms with cofibrant domain are homotopic just when isomorphic it will suffice to show that $v_{\f A,\f B}$ is an equivalence of categories.  To this end consider the composite:
$$\xy
(0,0)*+{\TAlgs(\f A,\f B)}="00"; (37,0)*+{\TAlgs(F1,[\f A,\f B])}="10";(78,0)*+{\Cat(1,\TAlg(\f A,\f B))}="20";(112,0)*+{\TAlg(\f A,\f B)}="30";
{\ar^<<<<<{v_{\f A,\f B}} "00";"10"};
{\ar^<<<<{\phi} "10";"20"};
{\ar^<<<<{ev_{\bullet}} "20";"30"};
\endxy$$
in which $\phi$ is the adjointness isomorphism -- recall that $U \circ [-,-] = \TAlg(-,-)$ -- and in which $ev_{\bullet}$ is the evaluation isomorphism.  It suffices to show that the composite is an equivalence.   $v_{\f A,\f B}$ sends $f:\f A \to \f B$ to $[\f A,\f f] \circ j:F1 \to [\f A,\f A] \to [\f A,\f B]$, whose image under $\phi$ is the functor $\TAlg(A,f) \circ \hat{1}_{\f A}:1 \to \TAlg(\f A,\f A) \to \TAlg(\f A,\f B)$.  Evaluating at $\bullet$ thus returns $f$ viewed as a pseudomap.  The action on 2-cells is similar and we conclude that the composite is the inclusion $\iota_{\f A,\f B}:\TAlgs(\f A,\f B) \to \TAlg(\f A,\f B)$.  As per Section~\ref{section:h2monads} this is an equivalence since $\f A$ is cofibrant.\\
Secondly we show that $$i_{\f A}:[F1, \f A] \to \f A$$ is a weak equivalence for all $\f A$: that its underlying functor $i_{\f A}:\TAlg(F1, \f A) \to A$ is an equivalence of categories.  Since $F1$ is cofibrant this is equally to show that the composite  $$i_{\f A} \circ \iota_{F1,\f A}:\TAlgs(F1,\f A) \to \TAlg(F1,\f A) \to A$$ is an equivalence.  An easy calculation shows that this is equally the composite 
$ev_{\bullet} \circ \phi:\TAlgs(F1,\f A) \to \Cat(1,A) \to A$
of the canonical adjunction and evaluation isomorphisms.  Hence $i_{\f A}$ is an equivalence for all $\f A$.\\
Let $u:\f A \to [\f B, \f A \otimes \f B]$ denote the unit of the adjunction $- \otimes \f B \dashv [\f B,-]$.  We are to show that the morphism $t_{\f A, \f B,\f C}$ given by the composite $$[u,1] \circ L:[\f A \otimes \f B,\f C] \to [[\f B, \f A \otimes \f B], [\f B, \f C]] \to [\f A,[\f B,\f C]]$$
is a weak equivalence for cofibrant $\f A$ and $\f B$.   Now the underlying functor of this composite is just the top row below.
\begin{equation*}
\cd{\TAlg(\f A \otimes \f B,\f C) \ar[r]^<<<<{[\f B,-]} & \TAlg([\f B,\f A \otimes \f B],[\f B,\f C]) \ar[rr]^{\TAlg(u,1)} && \TAlg(\f A,[\f B,\f C])\\
\TAlgs(\f A \otimes \f B,\f C)\ar[u]^{\iota} \ar[r]^<<<<{[\f B,-]} & \TAlgs([\f B,\f A \otimes \f B],[\f B,\f C]) \ar[u]^{\iota}\ar[rr]^{\TAlgs(u,1)} && \TAlgs(\f A,[\f B,\f C]) \ar[u]^{\iota}
}
\end{equation*}
In this diagram the left square commutes since $[\f B,-]$ restricts from $\TAlg$ to $\TAlgs$ and the right square since $u$ is a strict algebra map.  The outer vertical arrows are equivalences since both $\f A \otimes \f B$ and $\f A$ are cofibrant: the former using Axiom MC and the latter by assumption.  The bottom row is the adjointness isomorphism so that the top row is an equivalence by two from three.\\
Finally if $T$ is symmetric then, by Theorem~\ref{Theorem:SkewAlgebras}, the skew closed 2-category $(\TAlgs,[-,-],F1)$ is symmetric skew closed.  It now follows from Theorem~\ref{Theorem:hEKs} that $(\TAlgs,\otimes,F1)$ is homotopy symmetric monoidal.
\end{proof}

\subsubsection{The monoidal bicategory \TAlg}
A monoidal bicategory is a bicategory $\C$ equipped with a tensor product $\C \times \C \rightsquigarrow \C$ and unit $I$ together with equivalences $\alpha:(A \otimes B) \otimes C \to A \otimes (B \otimes C)$, $l:I \otimes A \to A$ and $r:A \to A \otimes I$ pseudonatural in each variable, and satisfying higher dimensional variants of the axioms for a monoidal category \cite{Gordon1995Coherence}.  Note that here we mean equivalences in the \emph{2-categorical or bicategorical sense}, as opposed to weak equivalences.\\  
In particular, each skew monoidal 2-category in which the components $\alpha$, $l$ and $r$ are equivalences provides an example of a monoidal bicategory.  The skew monoidal 2-category $(\TAlgs,\otimes,F1)$ is \emph{not} itself a monoidal bicategory.\begin{footnote}{ In fact $l:F1 \otimes \f A  \to \f A$ is an equivalence just when $\f A$ is equivalent to a flexible algebra.  Such algebras are called semiflexible \cite{Blackwell1989Two-dimensional}.}\end{footnote}
However $(\TAlgs,\otimes,F1)$ satisfies Axiom MC and hence, by Proposition~\ref{prop:axiomImplication}, it satisfies Axiom M.  Therefore the skew monoidal structure restricts to the full sub 2-category $\TAlgsc$ of cofibrant objects.  Since $(\TAlgs,\otimes,F1)$  is homotopy monoidal each component $\alpha$, $l$ and $r$ is a weak equivalence of cofibrant objects.  Since all objects are fibrant such weak equivalences are homotopy equivalences: thus equivalences in the 2-categorical sense.  We conclude:
\begin{Proposition}
The skew monoidal structure  $(\TAlgs,\otimes,F1)$ restricts to a skew monoidal 2-category $(\TAlgsc,\otimes,F1)$ which is a monoidal bicategory. 
\end{Proposition}
In fact $\TAlgsc$ is biequivalent to the 2-category $\TAlg$ of algebras and pseudomorphisms.  
\begin{Lemma}
The 2-adjunction $Q \dashv \iota:\TAlgs \leftrightarrows \TAlg$ restricts to a 2-adjunction $Q \dashv \iota:\TAlgsc 
\leftrightarrows \TAlg$ whose unit and counit are pointwise equivalences.  In particular, the composite inclusion $\iota:\TAlgsc \to \TAlg$ is a biequivalence.
\end{Lemma}
\begin{proof}
Because each $Q\f A$ is cofibrant/flexible the adjunction restricts.  The unit $q:\f A \rightsquigarrow Q\f A$ is an equivalence by Theorem 4.2 of \cite{Blackwell1989Two-dimensional}.   Since $\f A$ is flexible the counit $p_{\f A}:Q\f A \to \f A$ is an equivalence in $\TAlgs$ by Theorem 4.4 of \emph{ibid}.
\end{proof}
Just as monoidal structure can be transported along an adjoint equivalence of categories, so the structure of a monoidal bicategory may be transported along an adjoint biequivalence.  And we obtain the following result: see Theorem 14 of \cite{Hyland2002Pseudo}.  The  present argument has the advantage of dealing solely with the strict concepts of $\Cat$-enriched category theory until the last possible moment.
\begin{Theorem}
For $T$ an accessible pseudo-commutative 2-monad on $\Cat$ the 2-category $\TAlg$ admits the structure of a monoidal bicategory.
\end{Theorem}

\section{Bicategories}\label{section:homotopicalBicats}
We now return to the skew closed category $(\Bicat_{s},Hom,F1)$ of Section~\ref{section:bicategories} and show that it forms part of a monoidal skew closed category that is homotopy symmetric monoidal. A similar, but simpler, analysis yields the corresponding result for the skew structure on $\twocat_{s}$ discussed in Section~\ref{section:bicategories} -- this is omitted.

\subsection{Preliminaries on $\Bicat_{s}$}
To begin with, it will be helpful to discuss some generalities concerning homomorphism classifiers and the algebraic nature of $\Bicat_{s}$.\\
To this end, let us recall that the category $\CatGph$ of $\Cat$-enriched graphs is naturally a \emph{2-category} -- called $\f C\f G$ in \cite{Lack20082-nerves}.  $\f C\f G$ is locally presentable as a 2-category: that is, cocomplete as a 2-category and its underlying category $\CatGph$ is locally presentable.  Section 4 of \emph{ibid.} describes a filtered colimit preserving 2-monad $T$ on $\f C\f G$ whose strict algebras are the bicategories, and whose strict morphisms and pseudomorphisms are the strict homomorphisms and homomorphisms respectively.  The algebra 2-cells are called \emph{icons} \cite{Lack2010Icons}.  We write $\Icon_{s}$ and $\Icon_{p}$ for the corresponding extensions of $\Bicat_{s}$ and $\Bicat$ to 2-categories with icons as 2-cells.  It follows from \cite{Blackwell1989Two-dimensional} that the inclusion $\iota:\Icon_{s} \to \Icon_{p}$ has a left 2-adjoint $Q$: this assigns to a bicategory $A$ its \emph{homomorphism classifier} $QA$.\\
As mentioned $\CatGph$ is locally presentable.  Since $T$ preserves filtered colimits it follows that the category of algebras $\Bicat_{s}$ is locally presentable too, and that the forgetful right adjoint $U:\Bicat_{s} \to \CatGph$ preserve limits and filtered colimits.   Now the three functors from $\CatGph$ to $\Set$ sending a $\Cat$-graph to its set of (0/1/2)-cells respectively are represented by finitely presentable $\Cat$-graphs.  It follows that the composite of each of these with $U$ -- the functors $$(-)_{0}, (-)_{1}, (-)_{2}: \Bicat_{s} \to \Set$$ sending a bicategory to its set of (0/1/2)-cells -- preserves limits and filtered colimits. Now a functor between locally presentable categories has a left adjoint just when it preserves limits and is accessible: preserves $\lambda$-filtered colimits for some regular cardinal $\lambda$.  See, for instance, Theorem 1.66 of \cite{Adamek1994Locally}.  It follows that each of the above three functors has a left adjoint  -- we used the adjoint $F$ to $(-)_{0}$ to construct the unit $F1$ in Section~\ref{section:bicategories}.  

\subsection{The monoidal skew closed structure on $\Bicat_{s}$}
Our goal now is to show that $Hom(A,-):\Bicat_{s} \to \Bicat_{s}$ has a left adjoint for each $A$.  We will establish this by showing that $Hom(A,-)$ preserves limits and is accessible.  As pointed out above, the functors $(-)_{0}, (-)_{1}, (-)_{2}:\Bicat_{s} \to \Set$ preserve limits and filtered colimits.  Since they jointly reflect isomorphisms they also jointly reflect limits and filtered colimits.  Accordingly it will be enough to show that the three functors $$Hom(A,-)_{0}, Hom(A,-)_{1}, Hom(A,-)_{2}:\Bicat_{s} \to \Set$$ preserve limits and are accessible.  We argue case by case.
\begin{enumerate}
\item $Hom(A,B)_{0}$ is the set of homomorphisms from $A$ to $B$.  Hence $Hom(A,-)_{0}$ is naturally isomorphic to $\Bicat_{s}(QA,-)$ where $Q$ is the homomorphism classifier.  Like any representable functor $Hom(A,-)_{0}$ preserves limits and is accessible.
\item $Hom(A,B)_{1}$ is the set of pseudonatural transformations between homomorphisms. Let $Cyl(B)$ denote the following bicategory -- first constructed, in the lax case, in \cite{Benabou1967Introduction} .  The objects of $Cyl(B)$ are the morphisms of $B$ whilst morphisms $(r,s,\theta):f \to g$ are diagrams as below left
\begin{equation}\label{eq:Square}
\xy
(0,0)*+{a}="00";(15,0)*+{b}="10";
 (0,-15)*+{c}="01";(15,-15)*+{d}="11";
{\ar^{r} "00"; "10"}; 
{\ar_{f} "00"; "01"}; 
{\ar_{s} "01"; "11"}; 
{\ar^{g} "10"; "11"}; 
{\ar@{=>}^{\theta}(5,-7)*{};(10,-7)*{}}; 
\endxy
\hspace{2cm}
\xy
(0,0)*+{a}="00";(15,0)*+{b}="10";
 (0,-15)*+{c}="01";(15,-15)*+{d}="11";
{\ar^{r} "00"; "10"}; 
{\ar@/_1pc/_{f} "00"; "01"}; 
{\ar@/^1pc/^{f^{\prime}} "00"; "01"}; 
{\ar_{s} "01"; "11"}; 
{\ar@/^1pc/^{g^{\prime}} "10"; "11"}; 
{\ar@{=>}^{\alpha}(-2,-7)*{};(2,-7)*{}}; 
{\ar@{=>}^{\theta^{\prime}}(9,-7)*{};(14,-7)*{}}; 
\endxy
\hspace{0.1cm}
\xy
(0,-8)*+{=};
\endxy
\hspace{0.1cm}
\xy
(0,0)*+{a}="00";(15,0)*+{b}="10";
 (0,-15)*+{c}="01";(15,-15)*+{d}="11";
{\ar^{r} "00"; "10"}; 
{\ar@/_1pc/_{f} "00"; "01"}; 
{\ar_{s} "01"; "11"}; 
{\ar@/_1pc/_{g} "10"; "11"}; 
{\ar@/^1pc/^{g^{\prime}} "10"; "11"}; 
{\ar@{=>}^{\theta}(1,-7)*{};(6,-7)*{}}; 
{\ar@{=>}^{\beta}(13,-7)*{};(17,-7)*{}}; 
\endxy
\end{equation}
in which $\theta$ is invertible.  2-cells of $Cyl(B)$ consist of pairs of 2-cells $(\alpha,\beta)$ satisfying the equality displayed above right.  Note that here are strict projection homomorphisms $d,c:Cyl(B) \rightrightarrows B$ which, on objects, respectively select the domain and codomain of an arrow.\\
It is straightforward to see that we have a natural isomorphism of functors $Hom(A,-)_{1} \cong \Bicat(A,Cyl(-))$.  Combining this with $\Bicat(A,-) \cong \Bicat_{s}(QA,-)$  gives an isomorphism $Hom(A,-)_{1} \cong \Bicat_{s}(QA,Cyl(-))$.  Since this is the composite $\Bicat_{s}(QA,-) \circ Cyl(-):\Bicat_{s} \to \Bicat_{s} \to \Set$ whose second component is representable, it will suffice to show that $Cyl(-)$ preserves limits and is accessible.\\
For this, arguing as before, it is enough to show that each of $Cyl(-)_{0}$, $Cyl(-)_{1}$ and $Cyl(-)_{2}$ preserves limits and is accessible.  Certainly we have $Cyl(B)_{0} \cong B_{1}$ naturally in $B$ and $(-)_{1}$ preserves limits and filtered colimits.  We construct $Cyl(-)_{1}$ as a finite limit in four stages.  These stages correspond to the sets constructed below:
\begin{IEEEeqnarray*}{llu}
Opp(B_{2})=\{(x,y): x,y \in B_{2},  sx=ty,  tx=sy\}\\
Comp(B_{1})=\{(a,b):a,b \in B_{1},ta=sb\}\\
Iso(B_{2})=\{(x,y) \in Opp(B_{2}): y \circ x = i_{ty}, x \circ y = i_{tx}\}\\
Cyl(B)_{1}=\{(a,b,c,d,x,y):(x,y) \in Iso(B_{2}), \\
\qquad \qquad \qquad (a,b),(c,d) \in Comp(B_{1}),   sx = b \circ a, tx=d \circ c\}
\end{IEEEeqnarray*}
the first three of which, in turn, define the sets of pairs of 2-cells pointing in the opposite direction, of composable pairs of 1-cells, and of invertible 2-cells.  Each stage corresponds to the finite limit in $\CAT(\Bicat_{s},\Set)$ below.
\begin{equation*}
\cd{
Opp(B_{2}) \ar[r] & (B_{2})^{2} \ar@/^1.5ex/[r]^{(sx,tx)} \ar@/_1.5ex/[r]_{(ty,sy)} & (B_{1})^{2}
}
\end{equation*}
\begin{equation*}
\cd{Comp(B_{1}) \ar[r] \ar[d] & B_{1} \ar[d]_{s} & Iso(B_{2}) \ar[r] \ar[d] & Opp(B_{2})  \ar[d]_{(y \circ x,x \circ y)} & Cyl(B)_{1} \ar[r] \ar[d] & Iso(B_{2}) \ar[d]_{(sx,tx)} \\
B_{1} \ar[r]^{t} & B_{0} & (B_{1})^{2} \ar[r]^{i^{2}} & (B_{2})^{2} & Comp(B_{1})^{2} \ar[r]^<<<<<{(b \circ a,d \circ c)} & (B_{1})^{2}}
\end{equation*}
Now each of the functors $(-)_{0}$,$(-)_{1}$ and $(-)_{2}$ preserve finite limits and filtered colimits.  Since finite limits commute with limits and filtered colimits in $\Set$ it follows that each constructed functor, and in particular, $Cyl(-)_{1}$ preserves limits and filtered colimits.  Another pullback followed by an equaliser constructs $Cyl(-)_{2}$ and shows it to have the same preservation properties: we leave this case to the reader.
\item Finally observe that we can express $Hom(A,B)_{2}$ as the equaliser of the two functions $$Hom(A,Cyl(B))_{1} \rightrightarrows \Bicat(A,B)^{2}$$
 sending an element $\alpha:f \Rightarrow g$ of $Hom(A,Cyl(B))_{1}$ to the pair $(df,cf)$ and $(dg,cg)$ respectively.  This is natural in $B$.  Therefore $Hom(A,-)_{2}$ is a finite limit of functors, each of which has already been shown to preserve limits and be accessible.  Since finite limits in $\Set$ commute with limits and with $\lambda$-filtered colimits for each regular cardinal $\lambda$ it follows that $Hom(A,-)_{2}$ preserves limits and is itself accessible. 
\end{enumerate}
We conclude:
\begin{Proposition}
For each bicategory $A$ the functor $Hom(A,-):\Bicat_{s} \to \Bicat_{s}$ has a left adjoint $- \otimes A$.  In particular we obtain a monoidal skew closed category $(\Bicat_{s},\otimes, Hom,F1)$.
\end{Proposition}
\subsection{Homotopical behaviour of the skew structure}
We turn to the homotopical aspects of the skew structure. 
\subsubsection{The model structure on $\Bicat_{s}$}
A 1-cell $f:X \to Y$ in a bicategory $A$ is said to be an \emph{equivalence} if there exists $g:Y \to X$ and isomorphisms $1_{X} \cong gf$ and $1_{Y} \cong fg$.  Now a homomorphism of bicategories $F:A \rightsquigarrow B$ is said to be a biequivalence if it is essentially surjective up to equivalence (given $Y \in B$ there exists $X \in A$ and an equivalence $Y \to FX$) and locally an equivalence: each functor $F_{X,Y}:A(X,Y) \to B(FX,FY)$ is an equivalence of categories.\\
The relevant model structure on $\Bicat_{s}$ was constructed in \cite{Lack2004A-quillen}. The weak equivalences are those strict homomorphisms that are biequivalences.  A strict homomorphism  $F:A \to B$ is said to be a fibration if it has the following two properties (1) if $f:Y \to FX$ is an equivalence then there exists an equivalence $f^{\star}:Y^{\star} \to X$ with $Ff^{\star}=f$ and (2) each $F_{X,Y}:A(X,Y) \to B(FX,FY)$ is an isofibration of categories.  We note that all objects are fibrant.\\
The only knowledge that we require of the cofibrant objects is that each homomorphism classifier $QA$ is cofibrant.  To see this observe that if $f:A \to B$ is a trivial fibration then there exists a homomorphism $g:B \rightsquigarrow A$ with $f \circ g=1$.  Since the inclusion $\iota:\Bicat_{s} \to \Bicat$ sends each trivial fibration to a split epimorphism, and since split epis can be lifted through any object, an adjointness argument applied to $Q \dashv \iota$ shows that each $QA$ is cofibrant.  By Theorem 4.2 of \cite{Blackwell1989Two-dimensional} the counit $p_{A}:QA \to A$ is a surjective equivalence -- equivalence plus split epi -- in the 2-category $\Icon_{p}$.  Therefore $p_{A}$ is a trivial fibration and so exhibits $QA$ as a \emph{cofibrant replacement} of $A$.\\
The right homotopy relation on $\Bicat_{s}(A,B)$ is \emph{equivalence} in the bicategory $Hom(A,B)$.  Where needed, we will use the term \emph{pseudonatural equivalence} for clarity.  We note that a morphism $\eta:F \to G \in Hom(A,B)$ is an equivalence just when each component $\eta_{X}:FX \to GX$ is an equivalence in $B$.  That pseudonatural equivalence coincides with right homotopy follows from the fact, used in \emph{ibid.}, that the full sub-bicategory $PB$ of $Cyl(B)$, with objects the equivalences, is a path object for $B$.  In particular, if $A$ is a cofibrant bicategory then $F,G:A \rightrightarrows B$ are homotopic just when they are equivalent in $Hom(A,B)$.
\subsubsection{Homotopy monoidal structure}
Finally, we are in a position to prove the main theorem of this section.
\begin{Theorem}
The monoidal skew closed structure $(\Bicat_{s},\otimes,Hom,F1)$ satisfies Axiom MC and is homotopy symmetric monoidal.
\end{Theorem}
\begin{proof}
Firstly we show that the unit $F1$  is cofibrant.  Recall that $F$ is left adjoint to $(-)_{0}:\Bicat_{s} \to \Set$.  Since $(-)_{0}$ sends trivial fibrations to surjective functions, and since surjective functions can be lifted through $1$, it follows by adjointness that $F1$ is cofibrant.\\
In order to verify the remainder of Axiom MC we use the well known fact, see for example \cite{Street1980Fibrations}, that a homomorphism $F:A \rightsquigarrow B$ is a biequivalence if and only if there exists $G:B \rightsquigarrow A$ and equivalences $1_{A} \to GF$ and $1_{B} \to FG$.
A consequence is that if $F:A \to B$ is a biequivalence then so is $Hom(C,F)$ and $Hom(F,D)$ for all $C$ and $D$.\\
To verify Axiom MC, it remains to show that if $C$ is cofibrant and $F$ a fibration, then $Hom(C,F)$ is a fibration: in fact, we will show that this is true for all $C$.
To see that $Hom(C,F):Hom(C,A) \to Hom(C,B)$ is locally an isofibration, consider $\alpha:G \to H \in Hom(C,A)$ and $\theta:\beta \cong F\alpha$.  Then each component $\theta_{X}$ is invertible in $B$ and so lifts along $F$ as depicted below.
$$
\xy
(0,0)*+{GX}="00"; (30,0)*+{HX}="10"; 
{\ar@/^1.2pc/^{\beta^{\star}_{X}} "00";"10"};
{\ar@/_1.2pc/_{\alpha_{X}} "00";"10"};
{\ar@{=>}^{\theta^{\star}_{X}}(15,4)*+{};(15,-4)*+{}};
\endxy
\xy
(0,0)*+{}="00"; (10,0)*+{}="10"; 
{\ar@{|->}^{F} "00";"10"};
\endxy
\xy
(0,0)*+{FGX}="00"; (30,0)*+{FHX}="10"; 
{\ar@/^1.2pc/^{\beta_{X}} "00";"10"};
{\ar@/_1.2pc/_{F\alpha_{X}} "00";"10"};
{\ar@{=>}^{\theta_{X}}(15,4)*+{};(15,-4)*+{}};
\endxy
$$

The components $\beta^{\star}_{X}:GX \to HX$ admit a unique extension to a pseudonatural transformation $\beta^{\star}$ such that $\theta^{\star}:\beta^{\star} \to \alpha$ is  a modification: at $f:X \to Y$ the 2-cell ${\beta^{\star}}_{f}$ is given by:
$$\xy
(0,0)*+{Hf \circ \beta^{\star}_{X}}="00"; (30,0)*+{Hf \circ \alpha_{X}}="10"; (60,0)*+{\alpha_{Y} \circ Gf}="20";  (90,0)*+{\beta^{*}_{Y}\circ Gf}="30"; 
{\ar@{=>}^{Hf \circ \theta^{\star}_{X}} "00";"10"};
{\ar@{=>}^{\alpha_{f}} "10";"20"};
{\ar@{=>}^{(\theta^{\star}_{Y})^{-1} \circ Gf} "20";"30"};
\endxy$$
Then $F\beta^{\star}=\beta$ and we conclude that $Hom(C,F)$ is locally an isofibration.\\
It remains to show that $Hom(C,F)$ has the equivalence lifting property.  So consider $G:C \to A$ and an equivalence $\alpha:H \to FG \in Hom(C,B)$: a pseudonatural transformation with each component $\alpha_{X}:HX \to FGX$ an equivalence in $B$.  Since $F$ is a fibration there exists an equivalence $\beta_{X}:H^{\star}X \to GX \in A$ with $F\beta_{X}=\alpha_{X}$.  Each such equivalence forms part of an adjoint equivalence $(\eta_{x},\beta_{X} \dashv \rho_{x},\epsilon_{x})$ and at $f:X \to Y$ we define $H^{\star}(f):H^{\star}X \to H^{\star}Y$ as the conjugate
\begin{equation*}
\cd{
H^{\star}X \ar[r]^{\beta_{X}} & GX \ar[r]^{Gf} & GY \ar[r]^{\rho_{Y}} & H^{\star}Y
}
\end{equation*}
in which we take, as a matter of convention, this to mean $(\rho_{Y} \circ Gf) \circ \beta_{X}$.  With the evident extension to 2-cells $H^{\star}$ becomes a homomorphism.  Moreover the morphisms $\beta_{X}$ naturally extend to an equivalence $\beta:H^{\star} \to G \in Hom(C,A)$.\\
Although $FH^{\star}X=HX$ for all $X$ it is not necessarily the case that $Hf=FH^{\star}f$.  Rather, we only have invertible 2-cells $\phi_{f}:Hf \cong FH^{\star}f$ corresponding to the pasting diagram below.
$$\xy
(0,0)*+{HX}="00"; (25,10)*+{HY}="11";(25,-10)*+{FGX}="1-1"; (50,0)*+{FGY}="20";  (75,0)*+{HY}="30"; 
{\ar^{Hf} "00";"11"};
{\ar^{\alpha_{Y}} "11";"20"};
{\ar_{FGf} "1-1";"20"};
{\ar_{\alpha_{X}} "00";"1-1"};
{\ar_{F\rho_{Y}} "20";"30"};
{\ar@/^1.5pc/^{1_{Y}} "11";"30"};
{\ar@{=>}^{(\alpha_{f})^{-1}}(25,4)*+{};(25,-4)*+{}};
{\ar@{=>}^{F\eta_{Y}}(50,8)*+{};(50,3)*+{}};
\endxy$$
Indeed $\phi:H \cong FH^{\star}$ is an invertible icon in the sense of \cite{Lack2010Icons}.  Since $F$ is locally an isofibration these lift to invertible 2-cells $\phi^{\star}(f):H^{\star \star}(f) \cong H^{\star}f$.  Moreover $H^{\star \star}$ becomes a homomorphism, unique such that the above 2-cells yield an invertible icon $\phi^{\star}:H^{\star \star} \cong H^{\star}$.  Composing $\phi^{\star}:H^{\star \star} \cong H^{\star}$ and $\beta:H^{\star} \to G$ gives the sought after lifted equivalence.  This completes the verification of Axiom MC.\\
From Section~\ref{section:bicategories} we know that $(\Bicat_{s},Hom,F1)$ forms  a symmetric skew closed category.  According to Theorem~\ref{Theorem:hEKs} the skew monoidal $(\Bicat_{s},\otimes,F1)$ will form part of a homotopy symmetric monoidal category so long as $(\Bicat_{s},Hom,F1)$ is homotopy closed.\\
Firstly we show that $i=ev_{\bullet}:Hom(F1,A) \to A$ is a biequivalence for each $A$.  As pointed out in Section~\ref{section:bicategories} $F1$ has a single object and each parallel pair of 1-cells is connected by a unique invertible 2-cell.  Therefore the map $!:F1 \to 1$ is a biequivalence.  Hence $Hom(!,A):Hom(1,A) \to Hom(F1,A)$ is a biequivalence whereby it suffices to show that the composite $ev_{\bullet} \circ Hom(!,A)$ is a biequivalence.  This is just $ev_{\bullet}:Hom(1,A) \to A$.  It is straightforward, albeit tedious, to verify that this last map is a biequivalence directly.  For a quick proof we can use the fact that for each bicategory $A$ there is a strict 2-category $st(A)$ and biequivalence $p:A \rightsquigarrow st(A)$.  Since evaluation is natural in all homomorphisms the square below left commutes
\begin{equation*}
\cd{Hom(1,A) \ar[d]_{ev_{\bullet}} \ar@{~>}[rr]^{Hom(1,p)} && Hom(1,st(A)) \ar[d]_{ev_{\bullet}} & Ps(1,st(A)) \ar[l]_{\iota} \ar[dl]^{ev_{\bullet}} \\
A \ar@{~>}[rr]_{p} && st(A)}
\end{equation*}
and since both horizontal arrows are biequivalences it suffices to show that $ev_{\bullet}:Hom(1,st(A)) \to st(A)$ is a biequivalence.  Now let $Ps(1,st(A)) \to Hom(1,st(A))$ be the full sub 2-category containing the 2-functors.  It is easy to see that $\iota$ is essentially surjective up to equivalence -- $1$ \emph{is} a cofibrant 2-category! -- and hence a biequivalence.  Therefore we need only show that the composite $ev_{\bullet}:Ps(1,st(A)) \to st(A)$ is a biequivalence.  It is an isomorphism.\\
Finally we show that the function $v:\Bicat_{s}(A,B) \to \Bicat_{s}(F1,Hom(A,B))$ given by the composite
$$\xy
(0,0)*+{\Bicat_{s}(A,B)}="00"; (49,0)*+{\Bicat_{s}(Hom(A,A),Hom(A,B))}="10";(106,0)*+{\Bicat_{s}(F1,Hom(A,B))}="20";
{\ar^<<<<<{Hom(A,-)} "00";"10"};
{\ar^<<<<{\Bicat_{s}(j,1)} "10";"20"};
\endxy$$
is a bijection on homotopy classes of maps for each cofibrant $A$.  Firstly consider the strict homomorphism
\begin{equation*}
\cd{ Hom(A,B) \ar[r]^<<<<{L} & Hom(Hom(A,A),Hom(A,B)) \ar[r]^<<<<{Hom(j,1)} & Hom(F1,Hom(A,B))}
\end{equation*}
of bicategories.  By (C2) it composes with $i:Hom(F1,Hom(A,B)) \to Hom(A,B)$ to give the identity.  Since this last map is a biequivalence so too is $Hom(j,1) \circ L$ by two from three.  It follows that its underlying function $\Bicat(j,1) \circ Hom(A,-):\Bicat(A,B) \to \Bicat(F1,Hom(A,B))$ induces a bijection on equivalence classes of objects: pseudonatural equivalence classes of homomorphisms.\\
Now we have a commutative diagram
\begin{equation*}
\cd{\Bicat_{s}(A,B) \ar[d]_{\iota} \ar[rrrrr]^{ \Bicat_{s}(j,1) \circ Hom(A,-) } &&&&& \Bicat_{s}(F1,Hom(A,B)) \ar[d]^{\iota} \\
\Bicat(A,B) \ar[rrrrr]^{ \Bicat(j,1) \circ Hom(A,-)} &&&&& \Bicat(F1,Hom(A,B))}
\end{equation*}
in which the vertical functions are the inclusions.  Each of the four functions is well defined on pseudonatural equivalence classes: it follows, by two from three, that the top function will determine a bijection on pseudonatural equivalence classes if the two vertical inclusions do so.  More generally, if $X$ is a cofibrant bicategory the inclusion $\iota_{X,Y}:\Bicat_{s}(X,Y) \to \Bicat(X,Y)$ induces a bijection on pseudonatural equivalence classes.  For we can identify this inclusion, up to isomorphism, with the function $$\Bicat_{s}(p_{X},1):\Bicat_{s}(X,Y) \to \Bicat_{s}(QX,Y)$$
where $p_{X}:QX \to X$ is the counit of the adjunction $Q\dashv \iota$.  Since $p_{X}:QX \to X$ exhibits $QX$ as a cofibrant replacement of $X$, and so is a weak equivalence between cofibrant objects, it follows -- see, for instance, Proposition 1.2.5 of \cite{Hovey1999Model} -- that $\Bicat_{s}(p_{X},1)$ induces a bijection on homotopy classes, that is, pseudonatural equivalence classes, of morphisms.\\
From Section~\ref{section:bicategories} we know that $(\Bicat_{s},Hom,F1)$ forms a symmetric skew closed category.  Since it is homotopy closed we conclude from Theorem~\ref{Theorem:hEKs} that the skew monoidal $(\Bicat_{s},\otimes,F1)$ is homotopy symmetric monoidal.
\end{proof}

\section{Appendix}
\subsection{Proof of Theorem~\ref{Theorem:TotalDerived}.}
The isomorphism $$\phi^{d}:Ho(\C)(A \otimes_{l} B,C) \cong Ho(\C)(A,[B,C]_{r})$$ given by the composite 
\begin{equation*}
\def\objectstyle{\scriptstyle}
\def\labelstyle{\scriptstyle}
\cd{
Ho(\C)(QA\otimes QB,C) \ar[r]^<<<<<{(1,q)} & Ho(\C)(QA\otimes QB,RC) \ar[r]^{\phi} & Ho(\C)(QA,[QB,RC]) \ar[r]^{(p,1)^{-1}} & Ho(\C)(A,[QB,RC])}
\end{equation*}
is natural in each variable in $Ho(\C)$ since each component is natural in $\C$.\\
Now the left and right derived structures have components $(Ho(\C),\otimes_{l},I,\alpha_{l},l_{l},r_{l})$ and $(Ho(\C),[-,-]_{r},I,L_{r},i_{r},j_{r})$ respectively.  We must prove that these components are related by the equations ~\eqref{eq:leftunit}, \eqref{eq:rightunit} and \eqref{eq:ass} of Section~\ref{section:correspondence}.\\
For \eqref{eq:leftunit} we must show that the diagram
\begin{equation}\label{eq:vtriangle}
\cd{
Ho(\C)(A,B) \ar[r]^<<<<{(l_{l},1)} \ar[dr]_{v_{r}} & Ho(\C)(QI \otimes QA,B) \ar[d]^{\phi^{d}} \\
& Ho(\C)(I,[QA,RB])
}
\end{equation}
commutes for all $A$ and $B$. By naturality it suffices to verify commutativity in the case that $A$ is cofibrant and $B$ is fibrant.  By definition $v_{r}$ is the composite
\begin{equation*}
\def\objectstyle{\scriptstyle}
\def\labelstyle{\scriptstyle}
\cd{
Ho(\C)(A,B) \ar[r]^{Ho([QA,R-])} & Ho(\C)([QA,RA],[QA,RB]) \ar[r]^{([p,q],1)} & Ho(\C)([A,A],[QA,RB]) \ar[r]^{(j,1)} & Ho(\C)(I,[QA,RB]) 
}
\end{equation*}
Since $A$ is cofibrant and $B$ fibrant we can identify $Ho(\C)(A,B)$ with the set of homotopy classes $[f]:A \to B$ of morphisms $f:A \to B$; then $v_{r}([f])$ is the homotopy class of
\begin{equation*}
\cd{
I \ar[r]^{j} & [A,A] \ar[r]^{[p,q]} & [QA,RA] \ar[r]^{[1,Rf]} & [QA,RB]
}
\end{equation*}
which, by naturality of $p$ and $q$, coincides with the homotopy class of
\begin{equation*}\label{eq:shortpath}
\cd{
I \ar[r]^{j} & [A,A] \ar[r]^{[1,f]} & [A,B] \ar[r]^{[p,q]} & [QA,RB] & .
}
\end{equation*}
Therefore the shorter path in the diagram below is $v_{r}$.  The longer path below is, by definition, the longer path of the triangle 
\eqref{eq:vtriangle}.  Accordingly we must show that the following diagram commutes.
\begin{equation}\label{eq:vtriangleH}
\def\objectstyle{\scriptstyle}
\def\labelstyle{\scriptstyle}
\cd{
Ho(\C)(A,B) \ar[r]^{(p,1)} \ar[dr]_{(p,q)} \ar[dd]_{v} & Ho(\C)(QA,B) \ar[r]^{(l,1)} \ar[d]^{(1,q)} & Ho(\C)(I \otimes QA,B) \ar[d]_{(1,q)} \ar[r]^{(p \otimes 1,1)} & Ho(\C)(QI \otimes QA,B) \ar[d]^{(1,q)}\\
& Ho(\C)(QA,RB)\ar[r]^{(l,1)} \ar[dr]_{v} & Ho(\C)(I \otimes QA,RB) \ar[r]^{(p \otimes 1,1)} \ar[d]^{\phi} & Ho(\C)(QI \otimes QA,RB) \ar[d]^{\phi}\\
Ho(\C)(I,[A,B]) \ar[rr]^{(1,[p,q])} && Ho(\C)(I,[QA,RB]) \ar[r]^{(p,1)} \ar[dr]_{1} & Ho(\C)(QI,[QA,RB]) \ar[d]^{(p^{-1},1)}\\
& & & Ho(\C)(I,[QA,RB]
}
\end{equation}
Each object above is of the form $Ho(\C)(X,Y)$ for $X$ cofibrant and $Y$ fibrant and we view each $Ho(\C)(X,Y)$ as the set of homotopy classes of maps from $X$ to $Y$.  The morphisms are of two kinds.  Firstly there are those of the form $Ho(\C)(f,1)$ or $Ho(\C)(1,f)$ for $f$ a morphism of $\C$.  Such morphisms respect the homotopy relation and we view them as acting on homotopy classes.  The other morphisms are of the form $v$ or $\phi$ and, because $A$ is cofibrant and $B$ fibrant, each occurence is well defined on homotopy classes.  Accordingly, to verify that the above diagram commutes it suffices to verify that each sub-diagram commutes. 
Now apart from the commutative triangle on the bottom right, each sub-diagram of \eqref{eq:vtriangleH} consists of a  diagram involving the hom-sets of $\C$, but with components viewed as acting on homotopy classes.  Since in $\C$ itself these sub-diagrams commute, by naturality or \eqref{eq:leftunit}, they certainly commute when viewed as acting on homotopy classes.  Therefore  \eqref{eq:vtriangleH} commutes.\\
According to ~\eqref{eq:rightunit} we must show that 
\begin{equation}\label{eq:runit2}
\cd{Ho(\C)(QA \otimes QI,B) \ar[d]_{\phi^{d}} \ar[dr]^{(r_{l},1)} \\
Ho(\C)(A,[QI,RB]) \ar[r]_<<<<{(1,i_{r})} & Ho(\C)(A,B)
}
\end{equation}
commutes for each $A$ and $B$.  Note that in $Ho(\C)$ the morphism $1 \otimes e:QA \otimes I \to QA \otimes QI$ is inverse to $1 \otimes p:QA \otimes QI \to QA \otimes I$.  Accordingly we can rewrite $r_{l}$ as 
\begin{equation*}
\cd{
A \ar[r]^{p^{-1}} & QA \ar[r]^{r} & QA \otimes I \ar[r]^{1 \otimes e} & QA \otimes QI}
\end{equation*} 
by substituting $1 \otimes e$ for $(1 \otimes p)^{-1}$.  The following diagram then establishes the commutativity of \eqref{eq:runit2}.
\begin{equation*}
\def\objectstyle{\scriptscriptstyle}
\def\labelstyle{\scriptscriptstyle}
\cd{
Ho(\C)(QA \otimes QI,B) \ar[r]^{(1 \otimes e,1)} \ar[d]_{(1,q)} & Ho(\C)(QA \otimes I,B) \ar[d]_{(1,q)} \ar[dr]^{(r,1)}\\
Ho(\C)(QA \otimes QI,RB) \ar[d]_{\phi} \ar[r]^{(1 \otimes e,1)} & Ho(\C)(QA \otimes I,RB)  \ar[dr]^{(r,1)} \ar[d]_{\phi} & Ho(\C)(QA,B) \ar[dr]^{(p^{-1},1)} \ar[d]_{(1,q)} \\
Ho(\C)(QA,[QI,RB]) \ar[d]_{(p^{-1},1)} \ar[r]^{(1,[e,1])} & Ho(\C)(QA,[I,RB]) \ar[d]_{(p^{-1},1)} \ar[r]_{(1,i)} \ar[r] & Ho(\C)(QA,RB)  \ar[dr]_{(p^{-1},1)} & Ho(\C)(A,B) \ar[d]_{(1,q)} \ar[dr]^{1} \\
Ho(\C)(A,[QI,RB]) \ar[r]_{(1,[e,1])} & Ho(\C)(A,[I,RB])  \ar[rr]_{(1,i)} && Ho(\C)(A,RB) \ar[r]_{(1,q^{-1})} & Ho(\C)(A,B)
}
\end{equation*}
Next we calculate that $t^{d}:[Q(QA \otimes QB),RC] \to [QA,R[QB,RC]]$ as constructed in~\eqref{eq:t} has value:
\begin{equation*}
\def\objectstyle{\scriptstyle}
\def\labelstyle{\scriptstyle}
\cd{
[Q(QA \otimes QB),RC] \ar[r]^<<<<<{[p,1]^{-1}} & [QA \otimes QB,RC] \ar[r]^{t} & [QA,[QB,RC]] \ar[r]^{[1,q]} & [QA,R[QB,RC]] .}
\end{equation*}
This calculation is given overleaf by the commutative diagram ~\eqref{eq:CalculateT}.  All sub-diagrams of ~\eqref{eq:CalculateT} commute in a routine manner.  Apart from basic naturalities we use the defining equation $t = [u,1] \circ L$ of \eqref{eq:t}, the equation $[1,p] \circ k = p$ of ~\eqref{eq:k0} and naturality of $k$ as in ~\eqref{eq:k4}.  Furthermore, on the bottom right corner, we use that the morphisms $[Qp_{A},1], [p_{QA},1]:[QA,R[QB,RC]] \rightrightarrows [[QQA,R[QB,RC]]$ coincide in $Ho(\C)$.  To see that this is so we argue as in the proof of Lemma ~\ref{Theorem:LemmaK}.  Namely, $p_{QA}$ and $Qp_{A}$ are left homotopic because they are coequalised by the trivial fibration $p_{A}$ in $\C$, and, since $R[QB,RC]$ is fibrant, the desired equality follows.\\
Finally, we use the calculation of $t^{d}$ to prove that the diagram
\begin{equation*}
\cd{
Ho(\C)(QA \otimes Q(QB \otimes QC),D) \ar[dd]_{\phi^{d}} \ar[rr]^{Ho(\C)(\alpha_{l},1)} && Ho(\C)(Q(QA \otimes QB) \otimes QC,D) \ar[d]^{\phi^{d}}\\
&& Ho(\C)(QA \otimes QB,[QC,RD]) \ar[d]^{\phi^{d}}\\
Ho(\C)(A,[QB \otimes QC,RD]) \ar[rr]_{Ho(\C)(1,t^{d})} && Ho(\C)(A,[QB,R[QC,RD]])
}
\end{equation*}
instantiating ~\eqref{eq:ass} commutes for all $A$, $B$, $C$ and $D$.  This is established overleaf in the large, but straightforward, commutative diagram \eqref{eq:LargeAss} whose only non-trivial step is an application of ~\eqref{eq:ass} in $\C$ itself.

\thispagestyle{empty}
\begin{landscape}
\begin{equation}\label{eq:CalculateT}
\def\objectstyle{\scriptscriptstyle}
\def\labelstyle{\scriptscriptstyle}
\cd{
[Q(QA \otimes QB),RC] \ar[ddr]_{1} \ar[dddddd]_{[p,1]^{-1}} \ar[r]^{[Qq,1]^{-1}} & [QR(QA \otimes QB),RC] \ar[r]^{L} \ar[dd]^{[Qq,1]} & [[QB,QR(QA \otimes QB)],[QB,RC]] \ar[dd]^{[[1,Qq],1]} \ar[r]^{[k,1]} &  [Q[QB,R(QA \otimes QB)],[QB,RC]] \ar[dd]^{[Q[1,q],1]} \ar[r]^{[1,q]} & [Q[QB,R(QA \otimes QB)],R[QB,RC]] \ar[dd]^{[Q[1,q],1]}\\\\
&  [Q(QA \otimes QB),RC] \ar[r]^{L} & [[QB,Q(QA \otimes QB)],[QB,RC]] \ar[r]^{[k,1]} &  [Q[QB,QA \otimes QB],[QB,RC]] \ar[r]^{[1,q]}  \ar[dd]^{[Qu,1]} & [Q[QB,QA \otimes QB],R[QB,RC]] \ar[dd]^{[Qu,1]}\\\\
&& [[QB,QA \otimes QB],[QB,RC]] \ar[dd]^{[u,1]} \ar[uu]_{[[1,p],1]} \ar[uur]^{[p,1]} &  [QQA,[QB,RC]] \ar[r]^{[1,q]} & [QQA,R[QB,RC]]\\\\
 [QA \otimes QB,RC] \ar[uuuur]^{[p,1]} \ar[uurr]^{L} \ar[rr]_{t} &&  [QA,[QB,RC]] \ar[uur]^{[p,1]} \ar[rr]_{[1,q]}  && [QA,R[QB,RC]] \ar[uu]|{[p_{QA},1]=[Qp_{A},1]}\\\\
}
\end{equation}
\end{landscape}

\thispagestyle{empty}
\begin{landscape}
\begin{equation}\label{eq:LargeAss}
\def\objectstyle{\scriptstyle}
\def\labelstyle{\scriptstyle}
\cd{Ho(\C)(QA \otimes Q(QB \otimes QC),D)\ar[dd]_{(1,q)} \ar[r]^{((1 \otimes p)^{-1},1)} & Ho(\C)(QA \otimes(QB \otimes QC),D)  \ar[dd]_{(1,q)} \ar[r]^{(\alpha,1)} & Ho(\C)((QA \otimes QB) \otimes QC,D)\ar[d]^{(1,q)} \ar[r]^{(p \otimes 1,1)} & Ho(\C)(Q(QA \otimes QB) \otimes QC,D) \ar[d]^{(1,q)} \\
& & Ho(\C)((QA \otimes QB) \otimes QC,RD) \ar[dd]_{\phi}  \ar[r]^{(p \otimes 1,1)}  & Ho(\C)(Q(QA \otimes QB) \otimes QC,RD) \ar[d]^{\phi}\\
Ho(\C)(QA \otimes Q(QB \otimes QC),RD) \ar[dd]_{\phi} \ar[r]^{((1 \otimes p)^{-1},1)} & Ho(\C)(QA \otimes (QB \otimes QC),RD) \ar[dd]_{\phi} \ar[ur]^{(\alpha,1)} & & Ho(\C)(Q(QA \otimes QB),[QC,RD]) \ar[d]^{(p^{-1},1)} \\
&& Ho(\C)(QA \otimes QB,[QC,RD]) \ar[dd]_{\phi} \ar[ur]^{(p,1)} \ar[r]^{1} \ar[dr]_{(1,q)} & Ho(\C)(QA \otimes QB,[QC,RD]) \ar[d]^{(1,q)}\\
Ho(\C)(QA,[Q(QB \otimes QC),RD])  \ar[r]^{(1,[p,1]^{-1})} \ar[dd]_{(p^{-1},1)} &  Ho(\C)(QA,[QB \otimes QC,RD])\ar[dd]_{(p^{-1},1)} \ar[dr]_{(1,t)} & & Ho(\C)(QA \otimes QB,R[QC,RD]) \ar[d]^{\phi} \\
& & Ho(\C)(QA,[QB,[QC,RD]]) \ar[r]^{(1,[1,q])} \ar[d]^{(p^{-1},1)} & Ho(\C)(QA,[QB,R[QC,RD]]) \ar[d]^{(p^{-1},1)}\\
Ho(\C)(A,[Q(QB \otimes QC),RD]) \ar[r]_{(1,[p,1]^{-1})}  & Ho(\C)(A,[QB \otimes QC,RD]) \ar[r]_{(1,t)} & Ho(\C)(A,[QB,[QC,RD]]) \ar[r]_{(1,[1,q])} & Ho(\C)(A,[QB,R[QC,RD]])
}
\end{equation}
\end{landscape}

\end{document}